\documentclass[a4paper,11pt]{amsart}

\newfont{\cyr}{wncyr10 scaled 1100}

\usepackage[left=2.7cm,right=2.7cm,top=3.5cm,bottom=3cm]{geometry}
\usepackage{amsthm,amssymb,amsmath,amsfonts,mathrsfs,amscd, graphics}
\usepackage[latin1]{inputenc}
\usepackage[all]{xy}
\usepackage{latexsym}
\usepackage{longtable}
\usepackage{color}

\theoremstyle{plain}
\newtheorem{theorem}{Theorem}[section]
\newtheorem*{theorem*}{Theorem}
\newtheorem{corollary}[theorem]{Corollary}
\newtheorem{lemma}[theorem]{Lemma}
\newtheorem{proposition}[theorem]{Proposition}
\newtheorem{conjecture}[theorem]{Conjecture}
\newtheorem{assumption}[theorem]{Assumption}

\theoremstyle{definition}

\theoremstyle{remark}
\newtheorem{obswr}[theorem]{Observation}
\newtheorem{remarkwr}[theorem]{Remark}
\newtheorem{example}[theorem]{Example}

\newenvironment{remark}{\begin{remarkwr}\begin{upshape}}{\end{upshape}\end{remarkwr}}

\DeclareMathOperator{\alg}{alg}

\DeclareMathOperator{\fp}{\mathfrak{p}}

\DeclareMathOperator{\w}{w}

\DeclareMathOperator{\cts}{cts}

\DeclareMathOperator{\Ad}{Ad}

\DeclareMathOperator{\Ind}{Ind}

\newcommand{\hg}{{\mathbf{g}}}

\newcommand{\fu}{{\mathfrak{u}}}
\newcommand{\fv}{{\mathfrak{v}}}
\newcommand{\fw}{{\mathfrak{W}}}

\newcommand{\cW}{\mathcal W}

\newcommand{\V}{\mathcal V}

\newcommand{\Q}{\mathbb{Q}}

\newcommand{\Z}{\mathbb{Z}}

\newcommand{\C}{\mathbb{C}}

\newcommand{\Gal}{\mathrm{Gal\,}}
\newcommand{\GL}{\mathrm{GL}}

\newcommand{\Tr}{\mathrm{Tr}}

\newcommand{\Frob}{\mathrm{Fr}}
\newcommand{\End}{\mathrm{End}}

\newcommand{\Aut}{\mathrm{Aut}}

\newcommand{\ord}{{\mathrm{ord}}}
\newfont{\gotip}{eufb10 at 12pt}

\newcommand{\cO}{{\mathcal O}}

\newcommand{\cL}{{\mathcal L}}

\newcommand{\ra}{\rightarrow}
\newcommand{\lra}{\longrightarrow}

\DeclareMathOperator{\Hom}{Hom} 

\DeclareMathOperator{\cyc}{cyc}

\newcommand{\res}{\mathrm{res}}

 \DeclareMathOperator{\frob}{Fr}

\newcommand{\TT}{\mathbb{T}}

\def\lra{{\longrightarrow}}

\def\GL{{\bf GL}}

\def\rp1{r\!\!+\!\!1}

\def\psim{{\psi}}
\def\psig{{\psi_g}}

\include{thebibliography}

\begin{document}

\title[Deformations of weight one newforms]
{ First order $p$-adic deformations \\ of weight one newforms }

\author{Henri Darmon, Alan Lauder and Victor Rotger}

\begin{abstract}
This article  studies the   first-order
 $p$-adic deformations of  classical   weight one newforms,
 relating  their fourier coefficients to 
 the $p$-adic logarithms of algebraic numbers in the 
 field cut out by the associated projective Galois representation.
    \end{abstract}

\address{H. D.: McGill University, Montreal, Canada}
\email{darmon@math.mcgill.ca}
\address{A. L.: University of Oxford, U. K.}
\email{lauder@maths.ox.ac.uk}
\address{V. R.: Universitat Polit\`{e}cnica de Catalunya, Barcelona, Spain}
\email{victor.rotger@upc.edu}


\maketitle

\tableofcontents

\section*{Introduction}
 
 Let  $g$
 be a   classical cuspidal newform of weight one, level $N$ and nebentypus  character
$\chi: (\Z/N\Z)^\times \ra \Q_p^\times$, with fourier expansion
 $$g(q)  = \sum_{n = 1}^\infty a_n  q^n. $$
 The  {\em  $p$-stabilisations }  of $g$ 
 attached to a rational prime  $p\nmid N$ are the eigenforms of level $Np$ defined by
\begin{equation}
\label{eqn:p-stabilisations}
 g_\alpha(q):= g(q) - \beta \cdot g(q^p), \qquad 
 g_\beta(q) := g(q) - \alpha \cdot g(q^p),
 \end{equation}
 where $\alpha$ and $\beta$ are the (not necessarily distinct) roots of the 
 Hecke polynomial
$$ x^2-a_p  x + \chi(p) =: (x-\alpha)(x-\beta).$$
The forms $g_\alpha$ and $g_\beta$ are eigenvectors for the Atkin $U_p$ operator,
with eigenvalues $\alpha$ and $\beta$ respectively.  Since $\alpha$ and $\beta$ are
roots of unity, these eigenforms are both  {\em ordinary} at $p$.

An important feature of classical weight one forms is that they are associated to 
  odd, irreducible, two-dimensional Artin representations, 
     via a construction of Deligne-Serre.
   Let $\varrho_g:G_\Q\lra \GL_2(\C)$ 
   denote this Galois representation, and write $V_g$ for the underlying representation space.

A fundamental result of Hida asserts the existence of a {\em $p$-adic family } of ordinary eigenforms
specialising to $g_\alpha$ (or to $g_\beta$) in weight one.  
Bellaiche and Dimitrov \cite{BeDi}  later established the   uniqueness  of this 
 Hida family,   under the hypothesis that 
  $g$ is {\em regular at $p$} , i.e., that
 $\alpha \ne \beta$, 
 or equivalently,  that 
 the frobenius element at $p$ acts on $V_g$ with distinct eigenvalues. 
 In the intriguing special case where  $g$ is the theta series of a character
  of a real quadratic field $F$ in which 
the prime $p$ is split, the result of Bellaiche-Dimitrov further asserts that
the  unique ordinary
first-order  infinitesimal $p$-adic  deformation of $g$ 
is an overconvergent (but not classical) modular form
of weight one. 
  In \cite{DLR3}, the Fourier coefficients of this non-classical form
 were expressed as $p$-adic logarithms
of algebraic numbers in  a ring class field of $F$, suggesting that a closer examination of 
such deformations could have some relevance to explicit class field theory
for real quadratic fields.
 
The primary purpose of this  note  is extend the results of \cite{DLR3} to general
 weight one eigenforms.

Part  \ref{sec:regular} considers the regular setting where $\alpha\ne \beta$, in which
 the results exhibit a close analogy  to those of \cite{DLR3}. 

Part \ref{sec:irregular} takes up the case where $g$ is irregular at $p$. Here the
 results  are more fragmentary and less definitive. 
 Let $S_1^{(p)}(N,\chi)$ denote the  space of $p$-adic overconvergent modular forms
of  weight $1$, level $N$, and character $\chi$, and let 
 $S_1^{(p)}(N,\chi)[[g]]$ denote the generalised eigenspace
 attached to the system of Hecke eigenvalues of an irregular weight one form $g\in S_1(N,\chi)$. 
 The main conjecture of the
   second part  asserts that
$S_1^{(p)}(N,\chi)[[g]]$  is always four dimensional, with a two-dimensional subspace
 consisting of classical forms. Under this conjecture,  an 
explicit description of the elements of the generalised eigenspace
 in terms of their $q$-expansions
is provided.  
 The  resulting concrete description of the    generalised eigenspace 
that emerges from  Part \ref{sec:irregular} 
 is an indispensible
 ingredient in the extension of the ``elliptic Stark conjectures" of \cite{DLR1} to the 
 irregular setting that will be  presented in  \cite{DLR6}.

\medskip\medskip\noindent
{\small
 \thanks{{\bf Acknowledgements.} The first author was supported by an NSERC Discovery grant, and the
  third author was supported by Grant MTM2015-63829-P. The second author would like that thank
  Takeshi Saito of Tokyo University and Kenichi Bannai of Keio University for their hospitality. The three authors would like to thank the organizers of the Math Nisyros Conference in July 2017 for their hospitality before and after the workshop.
  This project has received funding from the European Research Council (ERC) under the European
Union's Horizon 
2020 research and innovation programme (grant agreement No 682152).
}}

\part{The regular setting} 
\label{sec:regular}

Let $\Lambda = \Z_p[[1+p\Z_p]]$ denote the    Iwasawa algebra, and  let 
$$ \cW := \Hom_{\cts}(1+p\Z_p,\C_p^\times) = \Hom_{\alg}(\Lambda, \C_p)$$ 
denote the associated {\em weight space}. For each $k\in \Z_p$, write
$\nu_k\in \cW$ for the 
``weight $k$"   homomorphism
   sending the group-like element $a\in 1+p\Z_p$ to 
$a^{k-1}$.  
The rule $\lambda(k) := \nu_k(\lambda)$ realises elements of $\Lambda$ as
analytic   functions on $\Z_p$. The spectrum 
$ \tilde \cW :=  \Hom_{\alg}(\tilde\Lambda, \C_p)$ of a 
 finite flat extension $\tilde\Lambda$ of $\Lambda$ 
is equipped with  a   ``weight map" 
$$ \w: \tilde\cW \lra \cW$$
of finite degree.
A $\Q_p$-valued point $x\in \tilde \cW$ is said to be {\em  of weight $k$} 
if $\w(x) = \nu_k$, and is said to be \'etale over $\cW$ if the inclusion 
$\Lambda\subset \tilde\Lambda$ induces an isomorphism between  
 $\Lambda$ and the completion of $\tilde\Lambda$ at the kernel 
 of $x$, denoted $\tilde\Lambda_x$.
An element of this
completion   thus gives rise to an analytic  function of $k\in \Z_p$ in a natural way.

A {\em Hida family}    is a 
formal $q$-series
$$  \hg := \sum a_n q^n \in \tilde\Lambda[[q]]$$
with coefficients in a finite flat extension $\tilde\Lambda$ of $\Lambda$, specialising to a classical 
ordinary eigenform of
weight $k$ at almost all  points $x$ of $\tilde \cW$ of weight $k\in \Z^{\ge 2}$.
Two Hida families $\hg_1 \in \tilde\Lambda_1[[q]]$ and $\hg_2\in \tilde\Lambda_2[[q]]$ 
are regarded as equal if the $\Lambda$-algebras 
$\tilde\Lambda_1$ and $\tilde\Lambda_2$ can
 be embedded in a common extension
 $\tilde\Lambda$ in such a way that $\hg_1$ and $\hg_2$ are identified.
 A well known theorem of Hida and Wiles asserts the existence of a Hida family specialising to $g_\alpha$ in weight one. 
 The following  uniqueness result for  this Hida family plays an important role in our study.
\begin{theorem*}
[Bellaiche, Dimitrov]
Assume that  the weight one form $g$ is regular at $p$, and let $x_\alpha$ and 
$x_\beta$ denote the distinct points on $\tilde\cW$ attached
to $g_\alpha$ and $g_\beta$  respectively.
Then
\begin{itemize}
\item[(a)] The curve $\tilde\cW$  is smooth at  $x_\alpha$ and $x_\beta$, and in particular 
 there are   unique Hida
families $\hg_\alpha, \hg_\beta \in \tilde \Lambda[[q]]$
specialising to $g_\alpha$  and $g_\beta$ at  $x_\alpha$ and $x_\beta$ respectively.
\item[(b)]
The weight map $\w:\tilde \cW \lra \cW$   is  furthermore 
 \'etale at $x_\alpha$  if any only if
$$
\quad \qquad (\dag) \quad \mbox{ $g$ is not the theta series of    a character of a real quadratic   $K$ in which  $p$
splits.} $$
\end{itemize}
\end{theorem*}
The setting where $g$ is regular at $p$ but 
$\w$ is not \'etale  at $x_\alpha$ has   been treated in 
\cite{DLR3}, and the remainder of Part A
 will therefore focus on the scenarios where $(\dag)$  is  
satisfied. In that case, after viewing elements of the completion $\tilde\Lambda_{x_\alpha}$ of 
$\tilde \Lambda$ at $x_\alpha$    as analytic functions of
the ``weight variable" $k$, one may consider  the  {\em canonical} $q$-series 
$$ g_\alpha' := \left(\frac{d}{dk} \hg_\alpha\right)_{k=1}$$
representing the first-order infinitesimal ordinary deformation of $\hg$ at the weight one point $x_\alpha$, 
along this canonical
``weight direction".  The $q$-series $g_\alpha'$ is analogous to the overconvergent 
generalised eigenform considered in \cite{DLR3}, with the following  differences:
\begin{itemize}
\item[(a)] While the overconvergent generalised eigenform of \cite{DLR3} is a (non-classical, but 
overconvergent) modular form of weight one, such an interpretation is not available for the 
$q$-series $g_\alpha'$, which should rather be viewed as the first order term of a ``modular form of weight
$1+\varepsilon$". 
\item[(b)] In the non-\'etale setting of \cite{DLR3}, the absence
 of a natural local coordinate with respect to which the derivative would be computed meant 
 that the overconvergent generalised eigenform of loc.cit. could only be meaningfully defined 
 up to scaling by a non-zero multiplicative factor. This ambiguity is not present in the definition
of $g_\alpha'$, whose fourier coefficients are therefore entirely well-defined.
\end{itemize}
The main results of Part A give explicit formulae for these fourier coefficients: they
 are stated in Theorems 
 \ref{thm:main-general-regular},
  \ref{thm:main-cm-p-split},
    \ref{thm:CM-p-inert}, and 
    \ref{thm:rm-case} below.

\section{The general case}
\label{sec:exotic-regular}

The goal of this section is to describe  a general formula 
for  the fourier coefficients of $g_\alpha'$.

\medskip
The Artin representation $V_g$ can be realised as a two-dimensional $L$-vector space, where 
 $L$ is a finite extension of $\Q$,  contained in a cyclotomic field. 
Let    $W_g = \hom(V_g, V_g)$ denote the adjoint  equipped with its usual conjugation action of  $G_\Q$, denoted
$$ \sigma \cdot w :=  \varrho_g(\sigma) w \varrho_g(\sigma)^{-1}, \qquad
\sigma \in G_\Q, \quad w\in W_g.$$
Let $H\subset H_g$ denote the finite Galois extensions of $\Q$ 
cut out by the representations $W_g$ and  $V_g$ respectively, and write
 $G:=\Gal(H/\Q)$.

For notational simplicity, the following assumption is made in the rest of this paper:
\begin{assumption}
The prime   $p$ splits completely in  the field $L$ of coefficients of the Artin representation $V_g$.
\end{assumption}
This assumption
 amounts to a simple congruence condition on $p$. The choice of an embedding of $L$ 
into $\Q_p$, which is
fixed henceforth, will allow us, when it is convenient, to  view 
 $V_g$  and $W_g$ as representations of $G_\Q$ with 
 coefficients in $\Q_p$, 
 and the 
  weight one form $g$ as  a modular form with 
  fourier  coefficients in $\Q_p$ rather than in $L$.
  The $\Q_p$-vector spaces $V_g$ and $W_g$ are thus equipped with   natural $G_\Q$-stable $L$-rational structures, denoted $V_g^L$ and
  $W_g^L$ respectively.

 An embedding of $\bar\Q$ into $\bar\Q_p$ is fixed once and for all, 
determining  a  prime  $\wp$    of $H$ and of $H_g$  above $p$,
and an associated 
 frobenius element $\tau_\wp$ 
in $\Gal(H_g/\Q)$ and in $G$. Let 
 $G_\wp\subset G$ be the decomposition subgroup 
 generated by $\tau_\wp$.

 The representations $V_g$ and $W_g$ admit the following decompositions as $\tau_\wp$-modules:
 $$ V_g = V_g^\alpha \oplus V_g^\beta, \qquad W_g = W_g^{\alpha\alpha} \oplus  W_g^{\alpha\beta} \oplus W_g^{\beta\alpha} \oplus W_g^{\beta\beta},$$
 where $V_g^\alpha$ and $\V_g^\beta$ denote the $\alpha$ and $\beta$-eigenspaces for the action of $\tau_\wp$ on $V_g$, 
 and 
 $$W_g^{\xi \eta} := \hom(V_g^\xi, V_g^\eta), \qquad \mbox{ for } \xi, \eta\in \{\alpha,\beta\}$$
 is a $G_\wp$-stable  line, on which $\tau_\wp$ acts with eigenvalue
 $\eta/\xi$.  
 Let 
 $$ W_g^{\ord} := \hom(V_g/V_g^\alpha, V_g) = W_g^{\beta\alpha} \oplus W_g^{\beta\beta}.$$ 
Of course, $W_g^{\ord}$ is stable under the action of $G_\wp$ but not under the action of $G$. 

We propose to give a general formula for the $\ell$-th fourier coefficient of $g_\alpha'$ as the trace of a certain explicit endomorphism of $V_g$, which is constructed via a series of lemmas. In the lemma below, we let $G$ act on $\cO_H^\times \otimes W_g$ diagonally  on both factors in the tensor product.
\begin{lemma}
\label{lemma:unit-one-dim}
The $\Q_p$-vector space $(\cO_H^\times \otimes W_g)^G$ of $G$-invariant 
vectors is one-dimensional.
\end{lemma}
\begin{proof}
Let $G_\infty$ be the subgroup of $G$ generated by a complex conjugation 
$c$, which has order two, since $V_g$ is odd.
By Dirichlet's unit theorem, 
the global unit group $\cO_H^\times\otimes \Q_p$  is isomorphic to $\Ind_{G_\infty}^ G(\Q_p) - \Q_p $ 
 as a  $\Q_p[G]$-module. 
 Let $W_g^0$ denote the three-dimensional 
representation of $G$ consisting of trace zero endomorphisms of $V_g$. As a representation of 
$G$, we have $W_g = W_g^0 \oplus \Q_p$, and $W_g^0$ does not contain the 
trivial representation as a constituent.  
 By Frobenius reciprocity,
 $$ \dim_{\Q_p}((\cO_H^\times \otimes W_g)^G) = 
 \dim_{\Q_p}((\cO_H^\times \otimes W_g^0)^G)  = \dim_{\Q_p}((W_g^0)^{c=1})  = 1.$$
 The result follows.
 \end{proof}
 
 Assume that the field  $L$  of coefficients   is large
  enough so that the semisimple ring $L[G]$ becomes 
  isomorphic to a direct sum 
  of matrix algebras over $L$. 
 The $L[G]$-module $\cO_H^\times\otimes L$ 
 decomposes as a direct sum of $V$-isotypic components,
 $$ \cO_H^\times \otimes L = \oplus_{V} \cO_H^\times[V],$$
 where $V$ runs over the irreducible representations of $G$, 
 and $\cO_H^\times[V]$ denotes
  the largest subrepresentation of $\cO_H^\times\otimes L$ which is isomorphic to a direct sum of copies of $V$ as 
 an $L[G]$-module. For a general, not necessarily irreducible, representation $W$, the module $\cO_H^\times[W]$ is defined as the direct sum of the $\cO_H^\times[V]$ as $V$  ranges over the irreducible constituents of $W$.
 Because $W_g$ (viewed, for now, as a representation with coefficients in $L$) is self-dual, 
 Lemma \ref{lemma:unit-one-dim} can be recast as the assertion that $\cO_H^\times[W_g]$ 
 is isomorphic to  a single irreducible subrepresentation of $W_g$.
 More precisely:
 
 \medskip \noindent
 $\bullet$ 
  In the case of ``exotic weight one forms" where
   $\varrho_g$ has non-dihedral projective image (isomorphic to $A_4$, $S_4$ or $A_5$), then 
\begin{equation}
\label{eqn:units-exotic}
\cO_H^\times[W_g]  = \cO_H^\times[W_g^0] \simeq W_g^0,
\end{equation}
 and hence is  three-dimensional.
 
 \medskip\noindent
 $\bullet$
 If $\varrho_g$ is induced from a character $\psi_g$ 
of an imaginary quadratic field $K$, then 
$$W_g = L \oplus L(\chi_K) \oplus V_\psi, $$
where $\chi_K$ is the odd quadratic Dirichlet character associated to 
$K$ and $V_\psi$ is the two-dimensional representation of $G$
induced from the ring class character
 $\psi = \psi_g/\psi_g'$  which cuts out the abelian extension $H$
 of $K$. The representation $V_\psi$  is irreducible if and only if $\psi$ is non-quadratic, and 
 in that case, 
\begin{equation}
\label{eqn:units-cm}
\cO_H^\times[W_g] = \cO_H^\times[V_\psi]\simeq V_\psi.
\end{equation}
 In the special case where $ \psi$ is quadratic, the representation 
 $V_\psi$ further  decomposes as the direct sum of one-dimensional representations attached to an even and  an odd quadratic Dirichlet character,
 denoted  $\chi_1$ and $\chi_{2}$ respectively. 
That special case, in which $V_g$ is also induced from a character of 
the real quadratic field cut out by $\chi_1$,
 is thus subsumed under
   \eqref{eqn:units-rm} below.
 
 \medskip\noindent
 $\bullet$
 If $\varrho_g$ is induced from a character $\psi_g$ 
of a real quadratic field
$F$, then 
$$W_g = L \oplus L(\chi_F) \oplus V_\psi, \qquad V_\psi := \Ind_F^\Q(\psi), \quad \psi:= \psi_g/\psi_g',  $$
and one always has
\begin{equation}
\label{eqn:units-rm}
\cO_H^\times[W_g] = \cO_H^\times[\chi_F] \simeq L(\chi_F),
\end{equation} 
i.e., $\cO_H^\times[W_g]$ is generated by a fundamental unit of $F$.

\medskip
Let $U_g^\times$ be any generator of the one-dimensional $\Q_p$-vector space
$(\cO_H^\times\otimes W_g)^G$ and let 
\begin{equation}
\label{eqn:def-Ug} 
 U_g := (\log_\wp \otimes {\rm id})(U_g^\times) \in H_\wp\otimes W_g
 \end{equation}
be the image of this vector  under the linear map
$$  \log_\wp \otimes {\rm id}: \cO_H^\times \otimes W_g \lra H_\wp \otimes W_g,$$
where $\log_\wp$ is the  $p$-adic logarithm on the $\wp$-adic completion 
$H_\wp$ of $H$ at $\wp$.
\begin{lemma}
\label{lemma:the-magic-A}
There exists a non-zero endomorphism $A\in H_\wp \otimes
W_g$ satisfying the following conditions:
\begin{enumerate}
\item[(a)]
\label{cond:galois-A}
${\rm Trace}(A U_g) = 0$. 
\item [(b)]
\label{cond:ordinary-A}
$A$ belongs to $H_\wp\otimes W_g^{\rm ord}$, i.e., 
$A(V_g^\alpha) = 0$.
\end{enumerate}
This endomorphism is unique up to scaling.
\end{lemma}
\begin{proof}
The space $H_\wp \otimes W_g$ is four-dimensional over $H_\wp$ and 
the conditions in  Lemma
\ref{lemma:the-magic-A}  amount to three 
 linear conditions on $A$. More precisely,
choose a $\tau_\wp$-eigenbasis $(v_\alpha, v_\beta)$ for $V_g$ for which 
$$\tau_\wp  v_\alpha = \alpha  v_\alpha, \qquad \tau_\wp v_\beta = \beta  v_\beta.$$
Relative to this basis,
the endomorphism $U_g$ is represented by a matrix of the form
$$ U_g :  \left(\begin{array}{cc} 
\log_{\wp}(u_1) & \log_{\wp}(u_{\beta/\alpha}) \\
\log_{\wp}(u_{\alpha/\beta}) & - \log_{\wp}(u_1) 
\end{array} \right),$$
where $u_1, u_{\alpha/\beta}$, and $u_{\beta/\alpha}$ are generators  of 
$\cO_H^\times[W_g]$ which (when non-zero)
    are eigenvectors for $\tau_\wp$, satisfying
$$ \tau_\wp(u_1) = u_1, \qquad \tau_\wp(u_{\beta/\alpha}) = (\beta/\alpha) u_{\beta/\alpha},  \qquad  \tau_\wp(u_{\alpha/\beta}) = (\alpha/\beta) u_{\alpha/\beta}. $$
The endomorphism $A$  satisfies  condition (b)
above if and only if  the 
 matrix representing it in the basis $(v_\alpha, v_\beta)$
is of the form
$$ A : \left( \begin{array}{cc} 0 & x \\ 0 & y \end{array}\right), \qquad x,y  \in H_\wp,$$
and condition (a)  
 implies the further linear relation
\begin{equation}
\label{eqn:lin-rel-A}
  \log_{\wp}(u_{\alpha/\beta}) \cdot x -  \log_\wp(u_1) \cdot y = 0.
  \end{equation}
  The injectivity of $\log_\wp: \cO_H^\times\otimes L \lra H_\wp$,  which follows from the linear independence over $\bar\Q$ of logarithms of algebraic numbers,
  implies that  
the coefficients $\log_{\wp}(u_{\alpha/\beta}) $ and $\log_\wp(u_1)$
 in \eqref{eqn:lin-rel-A}  vanish simultaneously if and only if
  $$u_{\alpha/\beta} =  u_1 = 0$$
  in $\cO_H^\times\otimes L$, i.e.,    
 if and only if
$\cO_H^\times[W_g]$ is one-dimensional  over $L$ and generated by
$u_{\beta/\alpha}$. 
This  immediately rules out 
\eqref{eqn:units-exotic} and \eqref{eqn:units-cm} as scenarios
for the structure of $\cO_H^\times[W_g]$, leaving only
  \eqref{eqn:units-rm}. Hence, $V_g$ is induced from a character of 
a real quadratic field
$F$. In that case, the lines spanned by
$u_{\alpha/\beta}$ and $u_{\beta/\alpha}$ are interchanged
under the action of any reflection in $G$, and hence
the condition  $u_{\alpha/\beta} =0$ implies  that
$u_{\beta/\alpha} =0$ as well, thus forcing the vanishing of the full 
  $\cO_H^\times[W_g]$. This contradiction 
  to Lemma \ref{lemma:unit-one-dim}
 shows that  \eqref{eqn:lin-rel-A} imposes a non-trivial linear condition
on $x$ and $y$, and therefore that $A$ is unique up 
to scaling.
  \end{proof}

\begin{lemma}
\label{lemma:bd-explicit}
Let $A$ be any element of $H_\wp\otimes W_g$ satisfying the conditions in Lemma
\ref{lemma:the-magic-A}.
Then the following   are equivalent:
\begin{itemize}
\item[(a)]
${\rm Trace}(A) \ne 0$;
\item[(b)] 
The representation $\varrho_g$ is not induced from a character of a real quadratic field in which the prime $p$ splits.
\end{itemize}
\end{lemma}

\begin{proof}
The vanishing of ${\rm Trace}(A)$ is equivalent to the vanishing of the entry
$y$ in \eqref{eqn:lin-rel-A}, and hence to the vanishing of
$\log_{\wp}(u_{\alpha/\beta})$, and therefore of $u_{\alpha/\beta}$ and
$u_{\beta/\alpha}$ as well. This implies that
$\cO_H^\times[W_g]$ is one-dimensional and generated by
$u_1$. As in the proof of Lemma \ref{lemma:the-magic-A},
this rules out \eqref{eqn:units-exotic} and 
\eqref{eqn:units-cm},   leaving  only \eqref{eqn:units-rm} as a 
possibility, i.e., $V_g$ is necessarily induced from a character of a real
quadratic field $F$.
Furthermore, $\tau_\wp$ fixes the group $\cO_H^\times[W_g]$ generated by the fundamental
unit of $F$, which  occurs precisely when
$p$ splits in $F$. The lemma follows.
\end{proof}

Assume from now on  that the equivalent conditions of Lemma 
\ref{lemma:bd-explicit} hold.  One can then define
 $A_g\in H_\wp \otimes W_g$ to be the unique $H_\wp^\times$- multiple of $A$ satisfying
$$ {\rm Trace}(A_g)=1.$$

As in lemma \ref{lemma:unit-one-dim}, $H_\wp \otimes W_g$ is endowed with the diagonal action of $G_\wp$ which acts on both
$H_\wp$ and on $W_g$ in a natural way. Given $A\in H_\wp\otimes W_g$ and $\sigma\in G_\wp$, let us write 
 $ ^{\sigma} A$  for the image of $A$ by the action of $\sigma$  on the first factor $H_\wp$, 
and   $\sigma \cdot A_g$ for the image of $A$ by the action of $\sigma$ by conjugation
 on the second
factor $W_g$. 
\begin{lemma}
\label{lemma:invariant-Gp}
The endomorphism $A_g$ belongs to the space
$(H_\wp \otimes W_g)^{G_\wp}$ of $G_\wp$-invariants for
the diagonal action of $G_\wp$ on $H_\wp \otimes W_g$,  i.e.,
$$  ^{\tau_\wp} A_g = \tau_\wp^{-1} \cdot A_g.$$ 
\end{lemma}
\begin{proof}
Relative to the $\Q_p$-basis for $V_g$ 
used in the proof of Lemma \ref{lemma:the-magic-A},
the endomorphism $A_g$ is represented by a matrix of the form
$$ \left( \begin{array}{cc} 
0 & \frac{\log_{\wp}(u_1)}{\log_{\wp}(u_{\alpha/\beta})} \\
0 & 1 
\end{array} \right).$$
The lemma follows immediately from this in light of the fact that conjugation 
by $\varrho_g(\tau_\wp)$ preserves the diagonal entries in such a matrix representation while multiplying its upper right hand entry by $\alpha/\beta$, whereas $\tau_\wp$ acts on the upper right-hand entry of the 
above matrix as multiplication by $\beta/\alpha$.
\end{proof}

The matrix $A_g$ gives rise to  a  
 $G$-equivariant homomorphism 
 $\Phi_g: H^\times \lra  H_\wp\otimes W_g$
 by setting
\begin{equation}
\label{eqn:def-Phi-g}
 \Phi_g(x) =   \sum_{\sigma \in G} \log_{\wp}(^\sigma x) \cdot (\sigma^{-1} \cdot A_g),
 \end{equation}
 where, just as above, the group $G$ acts on $H_\wp\otimes W_g$ trivially on the first factor and through the usual conjugation action induced by $\rho_g$ on the second factor.  
  \begin{lemma}
\label{lemma:about-image-Phig}
The homomorphism $\Phi_g$ takes values in $W_g$. 
\end{lemma}
\begin{proof}
For any $x\in (H\otimes \Q_p)^\times$ we have
\begin{eqnarray*}
 ^{\tau_\wp}\Phi_g(x)  &=& \sum_{\sigma \in G} \log_{\wp}(^{\tau_\wp\sigma} \!x) \cdot (\sigma^{-1} \cdot ^{\tau_\wp}\!\!A_g) \\
 &=& \sum_{\sigma \in G} \log_{\wp}(^{\tau_\wp\sigma}\! x) \cdot (\sigma^{-1} \cdot  {\tau_\wp}^{-1} \cdot A_g)  \\
 &=& \sum_{\sigma \in G} \log_{\wp}(^{\tau_\wp\sigma} \!x) \cdot ((\tau_\wp\sigma)^{-1} \cdot   A_g)  \\
 &=& \Phi_g(x),
 \end{eqnarray*}
 where Lemma \ref{lemma:invariant-Gp} has been used to derive the second equation.
 \end{proof}
 
  By a slight abuse of notation, we shall continue to denote with the same symbol the homomorphism 
 $$
 \Phi_g: (H\otimes \Q_p)^\times  \lra H_\wp\otimes W_g
 $$
obtained from \eqref{eqn:def-Phi-g} by extending scalars. Note that $H_\wp^\times$ embeds naturally in $(H\otimes \Q_p)^\times$.
 
  \begin{lemma}
  \label{lemma:cft-Phi-g}
The homomorphism $\Phi_g$ vanishes on $\cO_H^\times \otimes \Q_p$ and 
 $\Phi_g(H_\wp^\times) \subseteq H_\wp\otimes W_g^{\ord}$.
\end{lemma}

\begin{proof}
Picking $u\in \cO_H^\times$ and an arbitrary $B\in W_g$, set
$$U_g^\times := \sum_{\sigma\in G}  \phantom{x}^\sigma u \otimes (\sigma\cdot B) \in 
   (\cO_H^\times \otimes W_g)^G, \qquad U_g:= (\log_{\wp}\otimes {\rm id})(U_g) $$
 as in the statement of Lemma \ref{lemma:the-magic-A}. Note that $U_g^\times$ is either trivial or a generator of the one-dimensional space $(\cO_H^\times \otimes W_g)^G$.
We have
\begin{eqnarray*}
 {\rm Trace}(\Phi_g(u) \cdot B) &=&
{\rm Trace}\left( \left(\sum_{\sigma\in G} \log_\wp (^\sigma u) \cdot (\sigma^{-1} \cdot A_g)\right)\cdot B\right)
  \\ 
  &=& {\rm Trace}\left( A_g\cdot\left(\sum_{\sigma\in G} \log_\wp (^\sigma u) \cdot (\sigma \cdot B)\right) \right) \\
  &=& {\rm Trace}\left( A_g \cdot (\log_\wp\otimes {\rm Id}) (U_g^\times)\right)  = {\rm Trace}\left(A_g \cdot U_g\right).
   \end{eqnarray*}
 It follows from Property (a) satisfied by $A$ (and hence $A_g$ in particular)
   in Lemma \ref{lemma:the-magic-A}  that
   $$  {\rm Trace}\left( \Phi_g(u)\cdot B\right)  = 0, \qquad \mbox{ for all } B\in H_\wp\otimes W_g.$$
  The first assertion in the lemma   follows from the non-degeneracy of the $H_\wp$-valued
   trace pairing on $H_\wp\otimes W_g$.
   The second  assertion follows from Property (b) satisfied by  $A$ and  by $A_g$
 in
Lemma \ref{lemma:the-magic-A}.

  \end{proof}

  Let now $\ell\nmid Np$ be a rational prime, and let $\lambda$ be a prime of 
$H$ above $\ell$. Let $u(\lambda) \in \cO_H[1/\lambda]^\times \otimes \Q$ be a $\lambda$-unit of $H$  satisfying
\begin{equation}
\label{eqn:pin-down-u-lambda}
{\rm Norm}^H_\Q (u(\lambda)) = \ell.
\end{equation}
This condition makes $u(\lambda)$ well-defined up to the addition of 
elements in $\cO_H^\times\otimes \Q$, and hence the 
element
$$A_g(\lambda) := \Phi_g(u(\lambda)) =
\sum_{\sigma \in G} \log_{\wp}(^\sigma u(\lambda)) \cdot (\sigma^{-1} \cdot A_g)
$$ 
is
well-defined, by Lemma \ref{lemma:cft-Phi-g}. 
Although $A_g(\lambda)$ only belongs to  $H_\wp \otimes W_g$ a priori, we have:
\begin{lemma}
\label{lemma:about-Ag-lambda}
The trace of the endomorphism $A_g(\lambda)$ is equal to $\log_p(\ell)$. 
\end{lemma}
\begin{proof}
Since the trace of $A_g$ and its conjugates are all equal to $1$, we have
 \begin{eqnarray*}
 {\rm Trace}(A_g(\lambda))   &=& \sum_{\sigma \in G} \log_{\wp}( ^{\sigma} u(\lambda)) \cdot {\rm Trace}(\sigma^{-1}\cdot A_g) \\
&=& \sum_{\sigma \in G} \log_{\wp}(^{\sigma} u(\lambda)) \\
&=& \log_\wp\left({\rm Norm}^H_{\Q}(u(\lambda))\right).
\end{eqnarray*}
The latter expression is equal to $\log_p(\ell)$, by
 \eqref{eqn:pin-down-u-lambda}.
\end{proof}

 \begin{remark}
 Although $A_g(\lambda)$ belongs to $W_g$ by Lemma \ref{lemma:about-image-Phig}, 
the  entries of the matrix representing  
  $A_g(\lambda)$ relative to an $L$-basis 
 for $V_g^L$  are $L$-linear combinations of products
of $\wp$-adic logarithms of units and $\ell$-units in $H$, and in particular
$A_g(\lambda)$ need not lie in $W_g^L$. (In fact, it never does, since its trace is not
algebraic.)
\end{remark}

In addition to the invariant $A_g(\lambda)$, the choice of the prime $\lambda$ of $H$ above $\ell$ also 
 determines a well-defined
 Frobenius element 
$\tau_\lambda$ in $G=\Gal(H/\Q)$, and even in   
 $\Gal(H_g/\Q)$, since $\Gal(H_g/H)$ lies in the center of this group.
 
 We are now ready to state the main theorem of this section:
 \begin{theorem}
 \label{thm:main-general-regular}
 For all rational primes $\ell \nmid Np$,
 $$ a_\ell(g_\alpha') = {\rm Trace}( \varrho_g(\tau_\lambda) A_g(\lambda)).$$
 \end{theorem}
 \begin{remark}
 This invariant does not depend on the choice of a prime $\lambda$ of $H$ above $\ell$, since replacing $\lambda$ by another such prime has the effect of conjugating the endomorphisms $\varrho_g(\tau_\lambda)$ 
 and $A_g(\lambda)$ by the same element of $\Aut(V_g)$.
\end{remark}
 \begin{proof}[Proof of Theorem \ref{thm:main-general-regular}]
 Let  $\Q[\varepsilon]$ denote the ring of dual numbers over $\Q_p$,
 with $\varepsilon^2 = 0$, and let 
 $$\tilde\varrho_g: G_\Q \lra \GL_2(\Q_p[\varepsilon])$$ 
be the unique first order $\alpha$-ordinary  deformation of $\varrho_g$ satisfying
$$ \det\tilde\varrho_g = \chi_g (1+  \log_p \chi_{\cyc} \cdot \varepsilon).$$
 This representation  may be written as
\begin{equation}
\label{eqn:introducing-kappa}
 \tilde\varrho_{g} = (1+   \varepsilon \cdot \kappa_g)  \cdot \varrho_g \quad \mbox{ for some  } \quad \kappa_{g}: G_{\Q} \lra  W_g.
\end{equation}
 The multiplicativity of  $\tilde \varrho_{g}$ 
    implies that the function $\kappa_{g}$ is a
   $1$-cocycle on $G_{\Q}$ with values in $W_g$, whose class in $H^1(\Q,W_g)$
   (which shall be  denoted with the same symbol, by a slight abuse of notation) 
    depends only
   on the isomorphism class of $\tilde \varrho_{g}$. 
   Furthermore,  
   $$ a_\ell(g_\alpha) + \varepsilon \cdot a_\ell(g_\alpha') = {\rm Trace}(\tilde\varrho_g(\tau_\lambda)) = a_\ell(g) + \varepsilon \cdot {\rm Trace}(\kappa_g(\tau_\lambda) \varrho_g(\tau_\lambda)),$$
   and hence
 \begin{equation}
 a_\ell(g_\alpha') = {\rm Trace}(\varrho_g(\tau_\lambda) \kappa_g(\tau_\lambda)).
 \end{equation}
 To make $\kappa_g(\tau_\lambda)$ explicit,  observe that
  the inflation-restriction sequence combined with global class field theory for $H$ gives rise to a series of identifications
  \begin{eqnarray} 
\nonumber
 H^1(\Q,W_g)  & \stackrel{\res_H}{\lra} &  \hom(G_H, W_g)^G \\
 \nonumber
 &=& \hom_G\left( \frac{(\cO_H\otimes \Q_p)^\times}{\cO_H^\times\otimes\Q_p}, W_g\right). 
 \end{eqnarray}
Under this identification, the class $\kappa_g$  can be viewed as an 
element of the space
$$H^1_{\rm ord}(\Q,W_g) = \left\{ \Phi\in  \hom_G\left( \frac{(\cO_H\otimes \Q_p)^\times}{\cO_H^\times\otimes\Q_p}, W_g\right) \mbox{ such that } 
\Phi(H_{\wp}^\times)  \subset W_g^{\rm ord} \right\}.$$
But the homomorphism $\Phi_g$ of 
\eqref{eqn:def-Phi-g}   belongs to the same one-dimensional space, by
 Lemma  \ref{lemma:about-image-Phig} and  \ref{lemma:cft-Phi-g}.
By global class field theory,
 the endomorphism $\kappa_g(\tau_\lambda)$ is therefore 
 a  $\Q_p^\times$-multiple
  of $\Phi_g(u_g(\lambda)) = A_g(\lambda)$. 
The fact that these endomorphisms
 are actually equal now follows by comparing their
   traces and noting that
$$ {\rm Trace}(\kappa_g(\tau_\lambda)) = \log_p \chi_{\cyc}(\ell) = \log_p(\ell),
$$
while 
$$   {\rm Trace}(A_g(\lambda)) =  \log_p(\ell),$$
by Lemma \ref{lemma:about-Ag-lambda}.
Theorem \ref{thm:main-general-regular} follows.
 \end{proof}
 
 \begin{corollary}
 \label{cor:main-general-regular}
 If the rational prime $\ell\nmid Np$ splits completely in $H/\Q$, then 
 $$a_\ell(g_\alpha') = 
(1/2) \cdot a_\ell(g) \cdot \log_p(\ell).$$
 \end{corollary}
 \begin{proof}
 The hypothesis implies that $\varrho_g(\tau_\lambda)$ is a scalar, and
 hence that $\varrho_g(\tau_\lambda) = \frac{1}{2}a_\ell(g)$.
  It follows that
    $${\rm Trace}(\varrho_g(\tau_\lambda) A_g(\lambda)) = (1/2)\cdot a_\ell(g) \cdot {\rm Trace}(A_g(\lambda)) =(1/2)\cdot  a_\ell(g) \cdot \log_p(\ell).$$ 
    The corollary now follows from Theorem \ref{thm:main-general-regular}.
    \end{proof}

 \begin{example}
Let $\chi$ be a Dirichlet character of conductor $171$ with order $3$ at $9$ and $2$ at $19$.
Then $S_1(171,\chi)$ is a $\Q(\chi)$-vector space of dimension $2$. It is spanned by
an eigenform
\[ g = q + \zeta q^2 + \zeta^3 q^3 - \zeta^2 q^5 + (\zeta^2 - 1) q^6 + \cdots
\]
defined over $L := \Q(\zeta)$, with $\zeta$ a primitive $12$th root of unity, and its
Galois conjugate. (See \cite{BL} for all weight one eigenforms of level at most $1500$.)
The associated projective representation $\varrho_g$ has
$A_4$-image and factors through the field
\[ H = \Q(a),\, a^4 + 10 a^3 + 45 a^2 + 81 a + 81 = 0.\]
Let $p = 13$, which splits completely in $L$. The representation $\varrho_g$ is regular at $13$, with
eigenvalues $\alpha = \zeta$ and $\beta = - \zeta^3$. We computed the first order deformations through each of $g_\alpha$
and $g_\beta$ to precision $13^{10}$, and $q$-adic precision $q^{37,000}$, using methods based upon the algorithms
in \cite{lauder-alg}. 

The predictions made from Theorem  \ref{thm:main-general-regular} for $a_\ell(g_\alpha')$ depend upon the
conjugacy class of the Frobenius at $\ell$ in $\Gal(H/\Q)$. For all
 primes $\ell < 37,000$ which split  completely
in $H$, such as  $\ell = 109, 179, 449, 467, 521, \ldots$, we verified that  
\[   a_\ell(g_\alpha') = (1/2) \cdot a_\ell(g) \cdot \log_{13}(\ell)  \pmod{13^{10}},\]
as asserted by Corollary \ref{cor:main-general-regular}.
\end{example}

\section{CM forms}
\label{sec:cm-regular}
This section focuses on the case 
where  
$g = \theta_\psig$ is the CM theta series attached to 
a character 
$$\psig: G_K \lra L^\times$$ 
of a quadratic
imaginary  field $K$. The main theorems are
Theorems 
  \ref{thm:main-cm-p-split} and
    \ref{thm:CM-p-inert} below, which will be derived
     in two independent ways,
    both ``from first principles" and by specialising Theorem 
    \ref{thm:main-general-regular}.

 As in the previous section, the choice
 of an embedding of $L$ 
into $\Q_p$  allows us to view $\psi_g$ as a $\Q_p^\times$-valued character, and the 
weight one form $g$ as a modular form with coefficients in $\Q_p$. 

For a character $\psi:G_K\lra L^\times$, the notation $\psi'$ will be used to designate the composition of $\psi$ with conjugation by  the non-trivial element   in
$\Gal(K/\Q)$:
$$ \psi'(\sigma) = \psi(\tau\sigma\tau^{-1}),$$
where $\tau$ is any element of $G_\Q$ which acts non-trivially on $K$.

The Artin representation $\varrho_g$ is induced from $\psi_g$ and its restriction to $G_K$ is the direct 
sum $\psi_g \oplus \psi_g'$ of two characters of $K$, which are {\em distinct} by the irreducibility of
$\varrho_g$ resulting from the fact that $g$ is a cusp form.
In this case, the field
 $H$ is the ring class field of $K$ which is 
 cut out by  the non-trivial ring class character $\psim := \psig/\psig'$. 
The Galois group $G:= \Gal(H/\Q)$ is a generalised dihedral group containing 
$ Z := \Gal(H/K)$ as its abelian normal subgroup of index two.

The case   of CM forms can 
 be further subdivided into two sub-cases, depending on whether $p$ 
  is split or inert  in $K$.

\subsection{The case where $p$ splits in $K$}
Write $p \cO_K = \fp  \fp'$, and fix a prime $\wp$  of $\bar\Q$ above $\fp$. 
The roots of the $p$-th Hecke polynomial of $g$ are 
$$\alpha=\psi_g(\fp), \qquad \beta=\psi_g(\fp').$$
 This case  is notable in that the  Hida family $\hg$ passing through $g_\alpha$ can be written down explicitly 
as a family of theta series.
Its weight $k$ specialisation $\hg_k$  is the theta-series attached to the character $\psig \Psi^{k-1}$, 
where $\Psi$ is a CM Hecke character  of weight $(1,0)$ which is unramified at $\fp$. 
For all rational primes $\ell\nmid Np$, the  $\ell$-th fourier coefficient  of $\hg_k$ is given by
\begin{equation}
\label{eqn:theta-coeff-fam}
 a_\ell(\hg_k) = \left\{\begin{array}{ll} 
\psig(\lambda) \Psi^{k-1}(\lambda') + \psig(\lambda') \Psi^{k-1}(\lambda) & \mbox{ if } \ell = \lambda\lambda' \mbox{ splits in $K$};\\
0 & \mbox{ if $\ell$ is inert in $K$}. 
\end{array}\right.
\end{equation}
Letting $h$ be the class number of $K$ and  $t$  the cardinality of the unit group  $\cO_K^\times$, the character $\Psi$ satisfies
$$ \Psi(\lambda)^{ht} =  u_\lambda^t,  \qquad \mbox{ where } (u_\lambda) := \lambda^h,$$
for any prime ideal $\lambda$ of $\cO_K$ whose norm is the rational prime $\ell = \lambda\lambda'$.
Let $u_\lambda'$ denote the conjugate of $u_\lambda$ in $K/\Q$.
It follows  that     
\begin{eqnarray*}
 \frac{d}{dk} \Psi^{k-1}(\lambda)_{k=1} &=&   \log_{\fp}(u(\lambda)), \qquad \mbox{ where } u(\lambda) := u_\lambda \otimes \frac{1}{h} \in \cO_H[1/\ell]^\times \otimes \Q, 
 \end{eqnarray*}
 and likewise that
 \begin{eqnarray*}
 \frac{d}{dk} \Psi^{k-1}(\lambda')_{k=1} &=&   \log_{\fp}(u(\lambda)'), \qquad \mbox{ where } u(\lambda)':= u_\lambda' \otimes \frac{1}{h} \in \cO_H[1/\ell]^\times \otimes \Q.
 \end{eqnarray*} 
In light of \eqref{eqn:theta-coeff-fam}, we have obtained:
  \begin{theorem}
  \label{thm:main-cm-p-split}
For all rational primes $\ell$ that do not divide $Np$,
 \begin{equation}
\label{eqn:theta-coeff-fam-bis}
 a_\ell(g_\alpha') = \left\{\begin{array}{ll} 
\left(\psig(\lambda) 
\log_{\fp}(u(\lambda'))
 + \psig(\lambda')  
 \log_{\fp}(u(\lambda)) \right)
  & \mbox{ if } \ell = \lambda\lambda' \mbox{ splits in $K$};\\
0 & \mbox{ if $\ell$ is inert in $K$}. 
\end{array}\right.
\end{equation}
 \end{theorem}
 Thus, the prime fourier coefficients of $g_\alpha'$ are supported at the primes $\ell$ which are split in $K$,
 where they  are (algebraic multiples of)  the $\fp$-adic logarithms of $\ell$-units in this quadratic field.
  This general pattern will persist in the other settings to be described below, with the notable feature that the  fourier coefficients of $g_\alpha'$ will be more complicated 
  expressions involving, in general,  the $p$-adic logarithms of units and $\ell$-units 
 in the full  ring class field $H$.
 
The reader will note
 Theorem   \ref{thm:main-cm-p-split} is consistent with 
 Theorem  \ref{thm:main-general-regular}, and could also have been deduced from it.
 More precisely,   choose a basis of $V_g$ consisting of eigenvectors 
 for the action of $G_K$ (and hence also,  of $\tau_\wp$)  
  which are interchanged by some element $\tau\in G_\Q - G_K$. Relative to such a basis,
    the endomorphisms 
  $U_g$ and $A_g$ are represented by the following matrices, in which $u_\psi$ and $\tau u_\psi$ are
  generators of  the spaces of $\psi$ and $\psi^{-1}$-isotypic vectors in the group of elliptic units in 
  $\cO_H^\times\otimes L$:
  $$ U_g: \left(\begin{array}{cc}
  0 & \log_{\wp}(u_\psi) \\ 
  \log_{\wp}(u_\psi') & 0 
  \end{array}\right), \qquad
  A_g :  \left(\begin{array}{cc} 
  0 & 0 \\ 
  0 & 1 
  \end{array}\right).   $$
It follows that, if $\ell = \lambda\lambda'$ is split in $K$, then $A_g(\lambda)$ is represented by the matrix
  $$ A_g(\lambda) :   \left(\begin{array}{cc}
   \log_{\wp}(u(\lambda')) & 0  \\ 
  0 & \log_{\wp}(u(\lambda))  
  \end{array}\right), 
 $$
 while $A_g(\lambda) = \frac{1}{2} \log_p(\ell)$ is the scalar matrix with
  trace equal to $\log_p(\ell)$ if $\ell$ is inert in $K$.

 \subsection{The case where $p$ is inert in $K$}
We now turn to  the more interesting case where $p$ is inert in $K$.
Let $\sigma_\wp:= \tau_\wp^2$ 
denote the frobenius element in $G_K$ 
attached to the prime $\wp$ of $H$ (which is well-defined 
modulo the inertia subgroup at $\wp$). 
Note that
 the prime $p$ splits completely
in $H/K$, since the
 image of $\tau_{\wp}$ in $G$ is a reflection in this generalised dihedral group. 
 The
 image of $\sigma_\wp$  in $\Gal(H_g/K)$ therefore 
 belongs to the subgroup $\Gal(H_g/H)$ whose image under $\varrho_g$ 
 consists of scalar matrices. Similar notations and remarks apply to any rational prime $\ell$ which is inert
 in $K/\Q$.

Relative to an eigenbasis $(e_1,e_2)$  for the action of $G_K$ on $V_g$, the Galois representation $\varrho_g$ 
takes the form
\begin{equation}
\label{eqn:def-psi-g}
 \varrho_{g}(\sigma) = \left(\begin{array}{cc} 
\psi_g(\sigma) & 0 \\ 
0 & \psi_g'(\sigma) \end{array}\right) \ \ \mbox{for } \sigma\in G_K.
\end{equation}
The homomorphisms $\psi_g, \psi_g':  G_K\lra \Q_p^\times$  factor through
$\Gal(H_g/K)$ and satisfy
$$ \psi_g(\tau \sigma\tau^{-1}) = \psi_g'(\sigma), \quad \mbox{ for all } \tau \in G_\Q - G_K, \ \ \ \sigma\in G_K.$$ 
It follows that $\varrho_g(\tau)$ interchanges the lines spanned by
$e_1$ and $e_2$, for any element $\tau\in G_\Q-G_K$.
The restriction of $\varrho_g$ to $G_\Q - G_K$ can therefore be described in matrix form by 
 \begin{equation}
 \label{eqn:def-eta-g}
  \varrho_{g}(\tau) = \left(\begin{array}{cc} 
0 & \eta_g(\tau)   \\
  \eta_g'(\tau) & 0 \end{array}\right) \ \ \mbox{for } \tau \in G_\Q - G_K, 
\end{equation}
  where $\eta_g$ and $\eta_g'$ are $L$-valued functions  on $G_\Q - G_K$ 
  that satisfy 
  \begin{equation}
  \label{eqn:eta-tau1-tau2}
  \eta_g(\tau_1)\eta_g'(\tau_2) = \psi_g(\tau_1 \tau_2) = 
 \psi_g'(\tau_2\tau_1), \qquad \mbox{ for all } \tau_1,\tau_2\in G_\Q-G_K,
  \end{equation}
  as well as the relations
  \begin{equation}
  \label{eqn:eta-sigma-tau}
  \begin{array}{ll}
  \eta_g(\sigma\tau)=  \psi_g(\sigma) \eta_g(\tau),  &  \quad 
  \eta_g(\tau\sigma) = \psi_g'(\sigma) \eta_g(\tau),  \\
  \eta_g'(\sigma\tau) =   \psi_g'(\sigma) \eta_g'(\tau), &   \quad 
  \eta_g'(\tau\sigma) =   \psi_g(\sigma) \eta_g'(\tau),   
  \end{array}
     \qquad \mbox{ for all } \sigma\in G_K, \ \ \tau\in G_\Q-G_K.
     \end{equation}
  After re-scaling $e_1$ and $e_2$ if necessary, we may assume that 
  $\tau_\wp\in G_\Q - G_K$  is sent to the matrix
\begin{equation}
\label{frp-inert}   
 \varrho_g(\tau_\wp) =  \left(\begin{array}{cc} 
0 &  \zeta   \\ 
 \zeta & 0 
 \end{array}\right), \qquad \mbox{ with } \zeta^2 = - \chi_g(p).
\end{equation}
 The eigenvalues of $\varrho_g(\tau_\wp)$ are equal to $\alpha := \zeta$ and $\beta:=-\zeta$, and hence
 $g$ is {\em always regular} in this setting.

  Let 
  $$\tilde\varrho_g: G_{\Q} \lra \GL(\tilde V_g)$$ denote the first-order infinitesimal deformation
   of $\varrho_g$ attached to the Hida family $\hg$ passing through a choice of $p$-stabilization $g_\alpha$ of $g$,
  where  $\alpha\in \{\zeta,-\zeta\}$.
  The module $\tilde V_g$ is free of rank two over the ring
   $\Q_p[\varepsilon]/(\varepsilon^2) = \Q_p[[T]]/(T^2)$ 
arising from  the mod $T^2$ reduction of the representation $\varrho_{\hg}$ attached to  $\hg$.
Choose any $\Q_p[\varepsilon]$-basis $(\tilde e_1, \tilde e_2)$ of $\tilde V_g$
 lifting $(e_1,e_2)$,  and note that the restriction of $\tilde\varrho_g$  to $G_K$
is given by:
\begin{equation}
\label{trhog}
 \tilde\varrho_g(\sigma)  =  \left(
\begin{array}{cc} 
 \psig(\sigma)  \cdot (1+  \kappa(\sigma) \cdot   \varepsilon)  & \psi_g' (\sigma)  \cdot \kappa_\psi(\sigma)   \cdot \varepsilon \\
\psig(\sigma)   \cdot \kappa_\psi'(\sigma)  \cdot \varepsilon & \psi_g' (\sigma) \cdot (1+\kappa'(\sigma)  \cdot \varepsilon)  
\end{array} \right), \quad \mbox{ for all } \sigma\in G_K.
 \end{equation}
In this expression, 
\begin{itemize}
\item[(a)] The functions $\kappa$ and $\kappa'$ are continuous homomorphisms from $G_K$ to $\Q_p$, i.e., 
elements of $H^1(K,\Q_p)$,  which are interchanged by  conjugation by the involution in    $\Gal(K/\Q)$:
$$ \kappa(\tau\sigma\tau^{-1}) = \kappa'(\sigma), \qquad \tau\in G_\Q - G_K, \ \ \ \sigma\in G_K.$$
\item[(b)]
The functions $\kappa_\psi,\kappa_\psi':G_K \lra \Q_p$ 
are one-cocycles with values in $\Q_p(\psim)$, and 
give rise to well defined classes
$$ \kappa_\psi\in H^1(K, \Q_p(\psim)), \qquad 
\kappa_\psi' \in H^1(K, \Q_p(\psim^{-1})),$$  
which also satisfy
$$ \kappa_\psi(\tau\sigma\tau^{-1}) = \kappa_\psi'(\sigma), \qquad \tau\in G_\Q - G_K, \ \ \ \sigma\in G_K.$$
\end{itemize}

For each rational prime $\ell\nmid Np$, the  $\ell$-th 
fourier coefficient $a_\ell(g_\alpha')$  is given by
 \begin{equation}\label{fourier}
 a_\ell(g_\alpha') = \frac{d}{dk} \mbox{Trace}(\varrho_{\hg}(\tau_\lambda))_{k=1} 
 \end{equation}

Observe that the spaces $H^1(K,\Q_p)$ and $H^1(K, \Q_p(\psim))$ are
of dimensions two and one respectively over $\Q_p$, since  $\psim\ne 1$. 
More precisely, restriction to the inertia group at $p$ combined with 
local class field theory induces an isomorphism
\begin{equation}
\label{eqn:H1KQp}
 H^1(K,\Q_p) = \hom( \cO_{K_p}^\times, \Q_p) =  \Q_p \log_p(z)   \oplus \Q_p \log_p(z').
 \end{equation}
 
Let $\cO_H^{\times,\psi}$ denote the (one-dimensional) 
$\psi$-isotypic component of $\cO_H^{\times}\otimes \Q_p$ 
on which $\Gal(H/K)$ acts through the character $\psi$, and 
denote by $\wp$ the prime of $H$ above $p$ arising from our chosen embedding of
$\bar\Q$ into $\bar\Q_p$. 
Restriction   to the inertia group 
at $\wp$  in $G_H$ 
likewise gives rise  to an identification 
\begin{equation}
\label{eqn:H1KPsi}
 H^1(K,\Q_p(\psi)) = \hom( \cO_{H_{\wp}}^\times/ \cO_H^{\times,\psi}, \Q_p) = \Q_p\cdot( \log_\wp(u_\psi') \log_\wp(z) - \log_\wp(u_\psi) \log_\wp(z') ).
 \end{equation}
In the above equation, $u_\psi$ is to be understood as the natural image in $\cO_{H_{\wp}}^\times = \cO_{K_p}^\times$  of an  element of the form
$$  \sum_{\sigma\in Z} \psi^{-1}(\sigma) u^\sigma \in  (\cO_H^\times \otimes L)^{\psi},$$ 
where $u$ is an $L[G]$-module generator of $\cO_H^\times\otimes L$, 
and $u_\psi'$ is the image of $u_\psi$ under the conjugation action $K_p\lra K_p$. 
Note that replacing $u$ by $\lambda u$ for some $\lambda \in L[G]$  
has the effect of multiplying both $u_\psi$ and $u_\psi'$ by $\psi(\lambda) \in \Q_p$, 
so that the $\Q_p$-line spanned by the right-hand side
of \eqref{eqn:H1KPsi} is independent of the choice of $u\in \cO_H^\times$.

It follows  from \eqref{eqn:H1KQp} and
\eqref{eqn:H1KPsi} 
that the total deformation space 
of $\varrho_g$ (before imposing any  ordinarity hypotheses, or restrictions on the determinant)
is  three dimensional.

Let $v_g^+ := e_1 + e_2$ and  $v_g^-:= e_1-e_2$ be the eigenvectors for $\tau_\wp$ acting on $V_g$, with eigenvalues $\zeta$ and $-\zeta$ respectively.
Let  $\kappa_p$ and $\kappa_{\psi,\wp}$  denote the 
restrictions  $\kappa$ and $\kappa_\psi$  
to the inertia groups at $p$  and 
$\wp$ in $G_H$ and $G_K$respectively. Both can be viewed as 
characters of $K_p^\times = H_\wp^\times$ 
after identifiying the  abelianisations of $G_{K_p}$ and 
$G_{H_{\wp}}$  with a quotient of 
 $K_p^\times$ 
 via local class field theory.
\begin{lemma}
\label{lemma}
The following are equivalent:
\begin{itemize}
\item[(a)]
The inertia group at $\wp$ acts as the identity on some lift $\tilde v_g^+$ of $v_g^+$ to $\tilde V_g$;
\item[(b)]
The inertia group at $\wp$ acts as the identity on all lifts $\tilde v_g^+$ of $v_g^+$ to $\tilde V_g$;
\item[(c)]
The restrictions $\kappa_p$ and $\kappa_{\psi,\wp}$  satisfy
$$ \kappa_p(x) = -\kappa_{\psi,\wp}(x), \qquad 
\mbox{ for all } x\in \cO_{K_p}^\times.$$
\end{itemize}
Similar statements hold  when $v_g^+$ is replaced by $v_g^-$, where the conclusion is that $\kappa_p = \kappa_{\psi,\wp}$. 
\end{lemma}
\begin{proof}
The equivalence of the first two conditions follows from the fact that
 $ \varepsilon \tilde V_g \simeq V_g$ is unramified at $p$ and hence
that inertia acts as the identity on the kernel of the natural map $\tilde V_g \lra V_g$. 
To check the third, note that the inertia group $I_p$ at $p$ is contained in  $G_K$, since $K$ is unramified at $p$,
and observe that any  $\sigma\in I_p$  sends 
$\tilde e_1+ \tilde e_2$ to    
\begin{eqnarray*}
\tilde\varrho_g(\sigma) (\tilde e_1 + \tilde e_2) &=& 
 \tilde e_1 +\tilde e_2 + \varepsilon\cdot ( \kappa(\sigma) \tilde e_1 + \kappa_\psi'(\sigma) \tilde e_2 + \kappa_\psi(\sigma) \tilde e_1 + \kappa'(\sigma) \tilde e_2) \\ 
&=&   \tilde e_1 +\tilde e_2 + \varepsilon\cdot ( ( \kappa(\sigma)+  \kappa_\psi(\sigma))  \tilde e_1  + (\kappa'(\sigma) +  \kappa_\psi'(\sigma)) \tilde e_2).
  \end{eqnarray*}
  The lemma follows.
  \end{proof}
 A  lift $\tilde\varrho_g$   of $\varrho_g$ is ordinary relative to the space spanned by $v_g^+$ 
 if and only if it satisfies the equivalent conditions of Lemma \ref{lemma}.
This    lemma  merely spells out the proof of the 
   Bellaiche-Dimitrov theorem on   the one-dimensionality
  of the tangent space of the eigencurve at the point associated to $g_\alpha$. 
More precisely, the
   general ordinary first-order deformation of $\varrho_g$ is completely determined by the 
   pair $(\kappa_p, \kappa_{\psi,\wp})$, which depends on a  single linear parameter 
 $\mu \in\bar\Q_p$ and is given  by the rule
  \begin{eqnarray}
  \kappa_p(z) & =&   \mu (\log_\wp(u_\psi') \cdot \log_\wp(z) - \log_\wp u_\psi \cdot \log_\wp(z')), \\
 \kappa_{\psi,p}(z) &=& \pm \mu (\log_\wp(u_\psi') \cdot \log_\wp(z) - \log_\wp u_\psi \cdot \log_\wp(z')),
 \end{eqnarray}
  where the sign in the second formula
  depends on whether one is working with the ordinary deformation of  $g_\alpha$ or $g_\beta$.

 Let us now make use of the fact that
  $$\det(\tilde\varrho_g) = 1 + \varepsilon \log_p \chi_{\cyc}  = 1 +  \varepsilon \log_p(z  z').$$
  Since $\det(\tilde\varrho_g) = 1 + \varepsilon (\kappa+\kappa')$, this condition implies that
  $$ \mu =  \frac{1}{\log_\wp(u_\psi') - \log_\wp(u_\psi)},$$
  and hence that    $\kappa_p$  and $\kappa_{\psi,\wp}$ are  
  given by
 \begin{eqnarray}
 \label{eqn:apz}
  \kappa_p(z)  &=&  \frac{ \log_\wp(u_\psi') \cdot \log_\wp(z) - \log_\wp(u_\psi) \cdot \log_\wp(z')}{ \log_\wp(u_\psi') - \log_\wp(u_\psi)}, \\
  \label{eqn:kappapz}
  \kappa_{\psi,\wp}(z) &=& \pm \frac{\log_\wp(u_\psi')\cdot \log_\wp(z) - \log_\wp(u_\psi) \cdot \log_\wp(z')} {\log_\wp(u_\psi') - \log_\wp(u_\psi)}. 
  \end{eqnarray}
  Equations \eqref{eqn:apz} and \eqref{eqn:kappapz} give a 
completely explicit description of the first order
   deformation $\tilde\varrho_{g_\alpha}$ and $\tilde\varrho_{g_\beta}$, from which 
    the fourier coefficients of $g_\alpha'$ and $g_\beta'$ shall be  readily calculated.
    
 The formula for the $\ell$-th fourier coefficient of $g_\alpha'$ involves the unit $u_\psi$ above as well as certain $\ell$-units in 
 $\cO_H[1/\ell]^\times \otimes L$ whose definition depends on whether or not the prime
 $\ell$ is split or inert in $K/\Q$.
 
 If $\ell = \lambda \lambda'$ splits in $K/\Q$,   
 let $u(\lambda)$  and $u(\lambda')$ 
 denote, as before, the
     $\ell$-units in $\cO_K[1/\ell]^\times \otimes \Q$ of norm $\ell$ with prime factorisation 
     $\lambda$ and  $\lambda'$ respectively.
Set
          $$ u_g(\lambda) := u(\lambda)\otimes {\psig(\lambda)}  + u(\lambda')
          \otimes {\psig(\lambda')}, \qquad
    u_g(\lambda') :=  u(\lambda') \otimes {\psig(\lambda)}   +  u(\lambda) \otimes{\psig(\lambda')}.
    $$
    In other words, $u_g(\lambda)$ is the unique element of  $\cO_K[1/\ell]^\times \otimes L$ whose prime factorisation is 
    equal to  $ \psig(\lambda) \cdot \lambda + \psig(\lambda') \cdot \lambda'$. 
    Note that, if $\ell$ splits completely in $H/\Q$,
    i.e., if $\varrho_g(\tau_\lambda)$ is equal to  a scalar $\zeta$, 
    then $u_g(\lambda) = u_g(\lambda') =  \ell\otimes\zeta$, but that otherwise $u_g(\lambda)$ and
     $\ell$ generate the $L$-vector space
      $\cO_K[1/\ell]^\times\otimes L$  of $\ell$-units of $K$ (tensored with $L$). 

 If $\ell$ is inert in $K/\Q$, choose a prime $\lambda$ of $H$ lying above $\ell$, and let 
        $u(\lambda)  \in \cO_H[1/\lambda]^\times\otimes \Q$ be any
  $\lambda$-unit of $H$ satisfying 
 $ \ord_\lambda(u(\lambda)) = 1$, which is well defined
up to units in $\cO_H^\times$. Define the elements
\begin{eqnarray*}
u_\psi(\lambda) &=& \sum_{\sigma\in Z} \psim^{-1}(\sigma) \otimes \ ^\sigma \!u(\lambda)\quad\in\quad
L\otimes   \cO_H[1/\ell]^\times, \\
u_\psi'(\lambda) &=& \tau_\wp u_\psi(\lambda) \quad\in\quad
L\otimes   \cO_H[1/\ell]^\times.
\end{eqnarray*}
Thus $u_\psi(\lambda)$ lies in the $\psi$-component  $\cO_H[1/\ell]^\times[\psi]$ and  is  well-defined up to the addition of multiples of $u_\psi$,
where 
$$ u_\psi :=  \sum_{\sigma\in Z} \psim^{-1}(\sigma) \otimes \ ^\sigma\! u \quad\in\quad
L\otimes   \cO_H[1/\ell]^\times, $$
for any unit $u\in \cO_H^\times$,  while 
$u_\psi'(\lambda)$ lies in the $\psi^{-1}$ component and is well-defined up to the addition of multiples of 
$u_\psi'$, where
$$ u_\psi' = \tau_\wp  u_\psi.$$
Recall the function $\eta_g': G_\Q \setminus G_K$ introduced in \eqref{eqn:def-eta-g}, with values in the roots of unity of $L^\times$.    The main result of this section is:
  \begin{theorem}
  \label{thm:CM-p-inert}
 Let $\ell \nmid Np$ be a rational   prime. 
 \begin{itemize}
 \item[(a)]
  If $\ell = \lambda \lambda'$ splits in $K/\Q$, then
 \begin{equation}
  \label{eqn:CM-p-inert-l-split}
   a_\ell(g_\alpha') =   a_\ell(g_\beta') =  
    \frac{ \log_\wp(u_\psi') \cdot \log_\wp(u_g(\lambda))  - \log_\wp(u_\psi)\cdot \log_\wp  (u_g(\lambda'))}{ \log_\wp( u_\psi') - \log_\wp(u_\psi)}.
    \end{equation}
 \item[(b)]
  If $\ell$ 
   remains inert in $K/\Q$, then 
    $$ a_\ell(g_\alpha') =  \eta_g'(\tau_\lambda)  \frac{\log_\wp ( {u}_\psi')  \log_\wp (u_\psi(\lambda)) - \log_\wp (u_\psi ) \log_\wp ({u}_\psi'(\lambda)) }
    {\log_\wp (u_{\psi}') - \log_\wp(u_\psi))}.$$
    \end{itemize}
 \end{theorem}
\begin{proof}  
Let us first  compute first the fourier coefficients at primes $\ell \nmid Np$ that split
 as $\ell=\lambda \lambda'$ in $K$. Let $\sigma_\lambda$ and $\sigma_{\lambda'}$
 be the frobenius elements associated to 
 $\lambda$ and $\lambda'$ respectively. 
 They are well-defined elements in the Galois group of any abelian extension of $K$ in which 
 $\ell$ is unramified. 
 
 It follows from \eqref{fourier} and the matrix expression for $\tilde{\varrho}_{g|G_K}$ given in \eqref{trhog} that
\begin{eqnarray*}
 a_\ell(g_\alpha') &=& \psig(\lambda) \kappa(\lambda) + \psig(\lambda') \kappa(\lambda') \\ 
 &=& \psig(\lambda) \kappa_p(u(\lambda)) + \psig(\lambda') \kappa_p(u'(\lambda)).
 \end{eqnarray*}
 Equation \eqref{eqn:CM-p-inert-l-split}
 then  follows from the formula for $\kappa_p(z)$ given in 
 \eqref{eqn:apz}.
      
  We now  turn now to the computation of the fourier coefficients
   of $g'_\alpha$  at primes $\ell \nmid Np$ that remain inert in $K$. 
   Let $\tau_\lambda$ denote the Frobenius element in $\Gal(M/\Q)$ associated to the choice
   of a   prime ideal $\lambda$ above $\ell$ in $\bar\Q$, and let $\sigma_\lambda:= \tau_\lambda^2$
   denote the associated frobenius element in $\Gal(M/K)$. 
   
   Since $\tau_\lambda$ belongs to $G_\Q - G_K$, it follows from
   \eqref{eqn:def-eta-g} that the matrix $\tilde\varrho_g(\tau_\lambda)$ is of the form
   $$ \tilde\varrho_g(\tau_\lambda) = 
    \left(\begin{array}{cc}
   r_\ell \cdot \varepsilon &  \eta_g(\tau_\lambda)(1+ s_\ell\cdot \varepsilon) \\
   \eta_g'(\tau_\lambda)(1+t_\ell \cdot \varepsilon) & u_\ell  \cdot \varepsilon 
 \end{array} \right),$$
 for suitable scalars $r_\ell, s_\ell, t_\ell$, and $u_\ell \in \bar\Q_p$.  
  Since
$$
a_\ell(g_\alpha) + a_\ell(g_\alpha') \varepsilon= \mbox{Trace}(\tilde\varrho_{g}(\tau_\lambda)) = (r_\ell + u_\ell)\cdot\varepsilon,
$$
it follows that
\begin{equation}
\label{eqn:al-trace}
a_\ell(g_\alpha') = r_\ell + u_\ell.
\end{equation}
In order to compute this trace, 
we observe that it arises in the upper  right-hand and lower left-hand  entries of the matrix 
\begin{equation}
\label{eqn:lower-upper}
 \tilde\varrho_g(\sigma_\lambda) = \tilde\varrho_g(\tau_\lambda)^2 = 
   \left(\begin{array}{cc}
   \psi_g(\sigma_\lambda)(1 + (s_\ell + t_\ell)\cdot \varepsilon)     &  \eta_g(\tau_\lambda) (r_\ell+u_\ell) \cdot \varepsilon \\
   \eta_g'(\tau_\lambda)(r_\ell+u_\ell) \cdot \varepsilon  &  \psi_g(\sigma_\lambda) (1+ (s_\ell+t_\ell)\cdot \varepsilon) 
 \end{array} \right).
 \end{equation}
  On the other hand, since $\sigma_\lambda$ belongs to $G_K$ it follows from
   \eqref{trhog} that
\begin{equation}
\label{eqn:invoke-trhog}
 \tilde\varrho_g(\sigma_\lambda)  =  \left(
\begin{array}{cc} 
 \psig(\sigma_\lambda)  \cdot (1+ \kappa(\sigma_\lambda) \cdot   \varepsilon)  & \psi_g'(\sigma_\lambda)  \cdot \kappa_\psi(\sigma_\lambda)   \cdot \varepsilon \\
\psig(\sigma_\lambda)  \cdot \kappa_\psi'(\sigma_\lambda)  \cdot \varepsilon & \psi_g'(\sigma_\lambda)  \cdot (1+\kappa'(\sigma_\lambda)  \cdot \varepsilon).  
\end{array} \right)
\end{equation}
By comparing upper-right entries  in the matrices in 
\eqref{eqn:lower-upper} and  \eqref{eqn:invoke-trhog} and invoking 
\eqref{eqn:al-trace} together with the relation $\psi_g'(\sigma_\lambda) \eta_g(\tau_\lambda)^{-1} = \eta_g'(\tau_\lambda)$
arising from  \eqref{eqn:eta-tau1-tau2},
we deduce that 
$$
a_\ell(g_\alpha')  = \eta_g'(\tau_\lambda)\kappa_\psi(\sigma_\lambda).
$$
It is worth noting that   each of the expressions $\eta_g'(\tau_\lambda)$ and 
$\kappa_\psi(\sigma_\lambda)$ depend on the choice of a prime $\lambda$ of $H$ above $\ell$ that was made to 
define $\tau_\lambda$ and $\sigma_\lambda$, since changing this prime replaces
$\tau_\lambda$  and $\sigma_\lambda$ by their conjugates $\sigma\tau_\lambda \sigma^{-1}$
and $\sigma\sigma_\lambda \sigma^{-1}$ by   some $\sigma\in G_K$. More precisely, 
by \eqref{eqn:eta-sigma-tau} and the cocycle property of $\kappa_\psi$, 
$$ \eta_g'(\sigma \tau_\lambda \sigma^{-1}) = \psi^{-1}(\sigma) \eta_g'(\tau_\lambda), \qquad  \kappa_\psi(\sigma \sigma_\lambda \sigma^{-1}) = \psi(\sigma) \kappa_\psi(\sigma_\lambda).  
$$ 
In particular,  the product $\eta_g'(\tau_\lambda) \kappa_\psi(\sigma_\lambda)$  
 is independent of the choice of a prime above $\ell$, as it should be.
Note that $\eta_g'(\tau_\lambda)$ is a simple root of unity belonging to the image of $\psi_g$, while 
 $\kappa_\psi(\sigma_\lambda)$  represents the interesting ``transcendental"
 contribution to the fourier coefficient $a_\ell(g_\alpha')$. 
 
 By the description of $\kappa_\psi(\sigma_\lambda)$ arising from local and global class field theory, we conclude from
  \eqref{eqn:kappapz} that
\begin{equation}
\label{eqn:p-inert-l-inert}
a_\ell(g_\alpha') =  \eta_g'(\tau_\lambda) \frac{\log_\wp({u}_\psi') \log_\wp(u_\psi(\lambda)) - \log_\wp(u_\psi) \log_\wp({u}_\psi'(\lambda))}
{{\log_\wp (u_{\psi}')  - \log_\wp (u_\psi)}},
\end{equation}
as was to be shown.
\end{proof}

A more efficient (but somewhat less transparent) route to the proof of Theorem 
  \ref{thm:CM-p-inert} is to specialise Theorem \ref{thm:main-general-regular} 
  to this setting. Relative to a basis of the form $(v,\tau_\wp v)$ for $V_g$,
  where $v$ spans a $G_K$-stable subspace of $V_g$ on which $G_K$ acts via
  $\psig$, the matrix for $U_g$ is proportional to one
  of the form
  $$ U_g: \left(  \begin{array}{cc}
  0 & \log_{\wp}(u_\psi) \\
  \log_{\wp}(\tau_\wp u_\psi) & 0 \end{array}\right), $$
  and the ordinarity condition implies that the matrix representing $A_g$ is proportional
  to a matrix of the form 
  $$ A : \left( \begin{array}{cc} 
  x & -x \\ 
  y & -y 
  \end{array} \right). 
  $$
  The relations ${\rm Trace}(A_g U_g) = 0$ and ${\rm Trace}(A_g) =1$ 
  show that $A_g$ is represented by the matrix
  $$ A_g: \frac{1}{\log_\wp(u_\psi) - \log_\wp(\tau_\wp u_\psi)} \cdot \left( \begin{array}{cc} 
  \log_\wp(u_\psi) & -\log_\wp(u_\psi) \\
  \log_\wp(\tau_\wp u_\psi) & - \log_\wp(\tau_\wp u_\psi)
  \end{array} \right),
   $$ 
   and Theorem  
  \ref{thm:CM-p-inert} is readily deduced from the general formula for 
  the fourier coefficients of $g_\alpha'$ given in  Theorem \ref{thm:main-general-regular}. 
  The details are left to the reader.

\vspace{0.2cm}

\subsection{Numerical examples}
 
We begin with an illustration of  Theorem \ref{thm:CM-p-inert}
 in which the  image of $\varrho_g$  is isomorphic to the symmetric group
 $S_3$.
 
\begin{example}
Let $\chi$ be the quadratic character of conductor $23$ and 
\[ g  = q - q^2 - q^3 + q^6 + q^8 + \cdots \in S_1(23,\chi)\] 
be the theta
series attached to the imaginary quadratic field $K = \Q(\sqrt{-23})$. The Hilbert class
field $H$ of $K$ is 
\[ H = \Q(\alpha) \ \ \ \ \mbox{ where } \ \  \ \  \alpha^6 - 6 \alpha^4 + 9 \alpha^2 + 23 = 0.\]
Write $\Gal(H/K) = \langle \sigma \rangle$. The smallest prime  which is inert in $K$ is $p=5$.
The deformations $g^\prime_{1}$ and $g^\prime_{-1}$ were computed 
to a  $5$-adic precision of $5^{40}$ (and
$q$-adic precision $q^{600}$).

Consider the inert prime $\ell = 7$ in $K$. Let $u(7) = (2 \alpha ^4 - 7 \alpha ^2 + 5)/9$, a root
of $x^3 - x^2 + 2 x - 7 = 0$. Taking $\omega$ a primitive cube root of unity we have
\[ 
\begin{array}{rcl}
\log_{5}(u_\psi(7)) & =  & \log_5(u(7)) + \omega \log_5(u(7)^\sigma) + \omega^2 \log_5(u(7)^{\sigma^2}) \\
\log_{5}({u}'_\psi(7)) & =  & \log_5(u(7)) + \omega^2 \log_5(u(7)^\sigma) + \omega \log_5(u(7)^{\sigma^2}). 
\end{array}
\]
Let $u = (\alpha^2 - 1)/3$ be the elliptic unit in $H$, a root of $x^3 - x^2 + 1 = 0$. Then likewise we have
\[ 
\begin{array}{rcl}
\log_{5}(u_\psi) & =  & \log_5(u) + \omega \log_5(u^\sigma) + \omega^2 \log_5(u^{\sigma^2})  \\
\log_{5}({u}'_\psi)  & = &  \log_5(u) + \omega^2 \log_5(u^\sigma) + \omega \log_5(u^{\sigma^2}).
\end{array}
\]
Now
\[ a_7(g^\prime_{1}) = - a_7(g^\prime_{-1}) = 4083079847610157092272537548 \cdot 5 \bmod{5^{40}} \]
and one checks to $40$ digits of $5$-adic precision that
\[ a_7(g^\prime_{1}) = \frac{\log_5(u_\psi(7))\log_5({u}'_\psi) - \log_5({u}'_\psi(7)) \log_5(u_\psi)}{\log_5(u_\psi) - \log_5({u}'_\psi)}\]
as predicted by part (b) of  Theorem \ref{thm:CM-p-inert}.

Consider next the  prime $\ell = 13$, which splits  in $K$ and factors as
 $(l) = \lambda {\lambda'}$, where  $\lambda^3 = (u_\lambda)$ is a principal idead generated by
 $u_\lambda = -6 \alpha^3+ 18 \alpha - 37$, a root of $x^2 + 74 x + 2197$.
After setting $u(\lambda)  = u_\lambda \otimes \frac{1}{3}$, we let
\[
\begin{array}{rcl}
\log_5(u_g(\lambda)) & = &  \left( \omega \log_5(u(\lambda)) + \omega^2 \log_5({u}'(\lambda)) \right)\\
\log_5({u}_g'(\lambda)) & = & \frac{1}{3}  \left( \omega^2 \log_5(u(\lambda)) + \omega \log_5({u}'(\lambda)) \right).
\end{array}
\]
We have
\[ a_{13}(g^\prime_1) = a_{13}(g^\prime_{-1}) = -638894131680830198852008592 \cdot 5 \bmod{5^{40}} \]
and one sees that
\[   a_{13}(g^\prime_{\pm 1}) = \frac{ \log_5 ({u}'(\psi))  \log_5(u_g(\lambda)) - \log_5 (u(\psi))  \log_5( u_g'(\lambda))}{ \log_5({u}_\psi') - \log_5(u_\psi)}
\]
 to $40$ digits of $5$-adic precision,  confirming
 Part (a) of Theorem \ref{thm:CM-p-inert}.

 \end{example}

The experiment below focuses on the case where  $\psig$ is a quartic ring class character,
       so that $\varrho_g$ has image isomorphic to the dihedral group of order $8$. 
       The associated ring class character
        $\psi = \psig/\psi_g' = \psig^2$ of $K$ is   quadratic,
  i.e., a genus character which cuts out a biquadratic exension $H$  of $\Q$ containing $K$.
   Let $F$ denote the unique real quadratic 
  subfield of $H$, and let $K'$ 
   the unique imaginary quadratic subfield of $H$ which is distinct from $K$. 
  The unit $u_\psi$ is a power of the fundamental unit of $F$. 
  Observe that the prime $p$ is necessarily inert in $K'/\Q$, since 
  otherwise $\varrho_g$ would be induced from a character of
  the real quadratic field $F$ in which $p$ splits. 
   It follows that $u_\psi' = u_\psi^{-1}$, so that  by Theorem
  \ref{thm:CM-p-inert},
  $$
  a_\ell(g_\alpha') =  
    \frac{ -\log_p  u_\psi  \cdot( \log_p(u_g(\ell)) + \log_p(u_g'(\ell)))}{-2  \log_p(u_\psi) } =\frac{1}{2}\cdot( \log_p(u_g(\ell)) +\log_p(u_g'(\ell))).
$$
It follows from the definition of $u_g(\ell)$ that
$$ a_\ell(g_\alpha') = {\rm trace}(\varrho_g(\lambda)) \cdot \log_p(\ell).$$
In particular, we obtain
\begin{equation}
\label{eqn:alan-dihedral}
 a_\ell(g_\alpha') = \left\{ \begin{array}{ll} 
 \log_p(\ell) & \mbox{ if } \psig(\lambda)  = 1, \\
0 & \mbox{ if } \psig(\lambda) = \pm i, \\
- \log_p(\ell) & \mbox{ if } \psig(\lambda) = -1, \end{array} \right. 
\end{equation}
in perfect agreement with the  experiments below.

\begin{example}
Let $\chi = \chi_3 \chi_{13}$ where $\chi_3$ and $\chi_{13}$ are the quadratic characters of conductors
$3$ and $13$, respectively. The space $S_1(39,\chi)$ is one dimensional and spanned by the form
$ g = q - q^3 - q^4 + q^9 + \cdots $.
The representation $\rho_g$ has projective image $D_4$ and is induced from characters of two imaginary
quadratic fields and one real quadratic field. In particular, it is induced from the quadratic character $\psi_g$ of the Hilbert class field
\[ H = \Q(\sqrt{-39},a),\, a^4 + 4 a^2 - 48 = 0 \]
of $\Q(\sqrt{-39})$ (and also ramified characters of ray class fields of $\Q(\sqrt{-3})$ and $\Q(\sqrt{13})$).
Let $p = 7$, which is inert in $\Q(\sqrt{-39})$. We computed $g^\prime_{\pm 1}$ to $20$-digits of $7$-adic
precision (and to $q$-adic precision $q^{900}$).

First consider the
case of $\ell = \lambda \lambda^\prime$ split in $\Q(\sqrt{-39})$. Then one observes to $20$-digits of $7$-adic precision and all
such $\ell < 900$ that both $a_\ell(g_1')$ and $a_\ell(g_{-1}')$ satisfy
\eqref{eqn:alan-dihedral}.
Next we consider the case that $\ell$ is inert in $\Q(\sqrt{-39})$. Here one observes
numerically that the Fourier coefficients are zero when $\ell$ is inert in
$\Q(\sqrt{-3})$. When $\ell$ is split in $\Q(\sqrt{-3})$ the Fourier coefficients of the two stabilisations
are opposite in sign and both equal to 
the $p$-adic logarithm of a fundamental $\ell$ unit of norm $1$
in $\Q(\sqrt{-3})$. (Observe that
$p$ is split in $\Q(\sqrt{-3})$ and our numerical observations are
consistent in this example with both our theorems for $p$ split and $p$ inert
in the CM case.)
\end{example}

\section{RM forms}

Consider now the case where $g$ is the theta series attached to 
a character 
$$\psig: G_K \lra L^\times$$ 
of mixed signature of a real quadratic field $K$. As before,
assume for simplicity that the field $L$ may be embedded into $\Q_p$ and fix one such embedding. We also continue to denote $H_g$   
 the abelian extensions of $K$ which is  cut out by $\varrho_g$, and let 
 $H$ be the ring class field of $K$ 
 cut out by  the non-trivial
  ring class character $\psim := \psig/\psig'$. 
  Since $\psi_g$ has mixed signature, it follows that $\psi$ is totally odd and
   thus $H$ is totally imaginary.   As before, write  $G:= \Gal(H/\Q)$ and
   $Z:= \Gal(H/K)$.

As explained in the introduction, the case where $p$ splits in $K$ was already dealt with in
\cite{DLR3}, so  in this section
 we   only consider the case  where $p$ is inert in $K/\Q$. 
 The prime $p$ then splits completely
in $H/K$ and we fix a prime $\wp$ of $H$ above $p$. This choice determines an embedding 
$H \lra H_{\wp}= K_p$ and we write $z\mapsto {z'}$ for the conjugation action 
of $\Gal(K_p/\Q_p)$. 
Let $u_K\in \cO_K^\times$ denote the fundamental unit of $\cO_K$ of norm $1$, which
 we regard as an element of $K_p^\times=H_{\wp}^\times$ through the above embedding, and let $u_K' = u_K^{-1}$ denote its algebraic conjugate.

Let $(v_1,v_2)$ be a basis for $V_g$ consisting of eigenvectors for the action of 
$G_K$, and which are interchanged by the frobenius element $\tau_\wp$. 
Just as in the previous section, relative to this basis the Galois representation $\varrho_g$ 
takes the form
\begin{equation}
 \varrho_{g}(\sigma) = \left(\begin{array}{cc} 
\psi_g(\sigma) & 0 \\ 
0 & \psi_g'(\sigma) \end{array}\right)  \  \mbox{for } \sigma\in G_K, \qquad \varrho_{g}(\tau) = \left(\begin{array}{cc} 
0 & \eta_g(\tau)   \\
  \eta_g'(\tau) & 0 \end{array}\right) \  \mbox{for } \tau \in G_\Q - G_K,
\end{equation}
where $\eta_g$ and $\eta_g'$ are functions taking values in the group of roots of unity in $L^\times$.
The element $U_g\in (H_\wp\otimes W_g)$  of
\eqref{eqn:def-Ug}  
 is thus represented the matrix
$$ U_g : \left(\begin{array}{cc}
\log_{\wp}(u_K) & 0 \\
0 & \log_{\wp}(u_K') \end{array} \right),$$
and hence the endomorphism $A_g$ of Lemma \ref{lemma:invariant-Gp}
is represented by the particularly simple matrix
$$ A_g : \frac{1}{2} \left( \begin{array}{rr}
1 & -1 \\ 
-1 & 1 \end{array}\right).$$
It follows that, if $\ell=\lambda\lambda'$ is split in $K/\Q$, we have
\begin{equation}
\label{eqn:rm-pinert-lsplit}
 A_g(\lambda) : \frac{1}{2} \left( 
\begin{array}{cc}
\log_p(\ell) & 0 \\ 0 & \log_p(\ell) 
\end{array}\right),
\end{equation}
while if $\ell$ is inert in $K/\Q$ and $\lambda$ is a prime of $H$ lying above $\ell$,
\begin{equation}
\label{eqn:rm-pinert-linert}
A_g(\lambda) : \frac{1}{2} \left( 
\begin{array}{cc}
\log_p(\ell) & -\log_{\wp}(u_\psi(\lambda)) \\ -\log_{\wp}(u_\psi'(\lambda)) & \log_p(\ell) 
\end{array}\right).
\end{equation}
\begin{theorem}
\label{thm:rm-case}
 For all rational primes $\ell \nmid Np$,
 \begin{enumerate}
 \item[(a)] If $\ell$ is split in $K/\Q$, then
 $$ a_\ell(g_\alpha') =  \frac{1}{2} a_\ell(g) \cdot \log_p(\ell).$$
 \item[(b)]
 If $\ell$ is inert in $K/\Q$, then
 $$ a_\ell(g_\alpha') = -\left( \eta_g(\lambda) \log_{\wp}(u_\psi'(\lambda)) +
 \eta_g(\lambda') \log_{\wp}(u_\psi(\lambda)\right)).$$
 \end{enumerate} 
 \end{theorem}
\begin{proof}
This follows directly from  Theorem \ref{thm:main-general-regular}
in light of  \eqref{eqn:rm-pinert-lsplit}
and
\eqref{eqn:rm-pinert-linert}.
\end{proof}

\begin{remark}
As already remarked in \cite{DLR3},
Theorem
\ref{thm:rm-case} above (and also Theorem
\ref{thm-RMirreg} of Part B) display a striking analogy
with Theorem 1.1.~of \cite{duke-li} concerning the fourier expansions of mock modular forms
whose shadows are   weight one theta series  attached to characters of imaginary quadratic fields. 
The underlying philosophy is that 
 the $p$-adic deformations considered in this paper behave somewhat 
like mock modular forms of weight one,  ``with $\infty$ replaced by $p$". This explains
why 
the analogy remains compelling when 
the quadratic imaginary fields of \cite{duke-li} are  replaced by real quadratic fields in which $p$
is inert (these fields being  ``imaginary" from a $p$-adic perspective).
\end{remark}

We illustrate Theorem  \ref{thm:rm-case} 
on the form of smallest level whose associated
 Artin representation is induced from a character
 of a real quadratic field, but of  no imaginary quadratic field. The projective image in this example is the dihedral group  $D_8$ of order $8$:
\begin{example}
Let $\chi = \chi_5 \chi_{29}$ where $\chi_5$ and $\chi_{29}$ are quadratic and quartic characters of conductor $5$ and
$29$, respectively. Then $S_1(145,\chi)$ is one dimensional and spanned by the eigenform
\[ g = q + i q^4 - i q^5 + i q^9 + (-i - 1) q^{11} - q^{16} + (-i - 1) q^{19} + \cdots .\]
The form $g$ is induced from a quartic character of a ray class group of $K = \Q(\sqrt{5})$ (see \cite[Example 4.1]{DLR1} for a further discussion on this form). 
The relevant ring class field $H$ is
\[ H = \Q(\alpha) \mbox{ where }  \alpha^8 - 2 \alpha^7 + 4 \alpha^6 - 26 \alpha^5 + 94 \alpha^4 - 212 \alpha^3 + 761 \alpha^2 - 700 \alpha + 980  = 0. \]
Write $\Gal(H/K) = \langle \psi \rangle$. Take $p = 13$, and note that $\chi(p) = 1$ and so $\alpha = i$ and $\beta= -i$. We compute $g^\prime_{\pm i}$ to
$10$ digits of $13$-adic precision (and $q$-adic precision $q^{28,000}$).

Consider first the prime $\ell = 7$ which is inert in $K$.
We take the $7$-unit $u(7) \in H$ to satisfy $x^4 + 13  x^3 + 38 x^2 + 5 x + 343 = 0$ and define
\[ \log_{13}(u(7,\pm i)) := \log_{13}(u(7)) \mp i \log_{13}(u(7)^\psi) \mp \log_{13}(u(7)^{\psi^2}) \pm i \log_{13}(u(7)^{\psi^3})\]
and
so $\log_{13}(u(7,i)) + \log_{13}(u(7,-i)) = \log_{13}(v)$ where $v = u(7) / u(7)^{\psi^2} \in H$. 
Then one checks that to $10$ digits of $13$-adic precision
\[ a_7(g^\prime_i) = - \frac{1}{6} \cdot  \log_{13}(v)\]
which is in line with Theorem \ref{thm:rm-case}.
Next we take the prime $\ell = 11$ which is split in $K$. Then to $10$-digits of $13$-adic precision
\[ a_{11}(g^\prime_i) =  - \frac{i + 1}{2} \cdot  \log_{13}(11)\]
exactly as predicted by Theorem \ref{thm:rm-case}.
\end{example}

\part{The irregular setting}
\label{sec:irregular}

Denote by  $S_k(Np,\chi)$ (resp.~ by $S_k^{(p)}(N,\chi)$) the space  of classical (resp.~$p$-adic overconvergent) modular forms of weight $k$, level $Np$ (resp.\,tame level $N$) and character $\chi$, with coefficients in $\Q_p$. 
 The Hecke algebra  $\TT$ of level $Np$ generated over $\Q$ by the operators $T_\ell$
 with $\ell\nmid Np$ and $U_\ell$ with $\ell \mid Np$ acts naturally on the spaces
 $S_k(Np,\chi)$ and $S_k^{(p)}(N,\chi)$. 
 
As in the introduction, let $g\in S_1(N,\chi)$ be a newform and let $g_\alpha\in S_1(Np,\chi)$ be a $p$-stabilisation of $g$. The eigenform $g_\alpha$ gives rise to a  ring 
 homomorphism $\varphi_{g_\alpha}:\TT\lra L$ to the field $L$ generated by the fourier coefficients
 of $g_\alpha$, satisfying
\begin{equation}\label{correspondence} 
 \varphi_{g_\alpha}(T_\ell) = a_\ell(g_\alpha) \mbox{ if } \ell\nmid Np,
 \qquad \varphi_{g_\alpha}(U_\ell) = \left\{
 \begin{array}{ll}  
 a_\ell(g_\alpha)  & \mbox{ if } \ell \mid N;\\
\alpha & \mbox{ if } \ell = p.
 \end{array} \right.
\end{equation}
For any  ideal $I$ of a ring  $R$ and any $R$-module $M$, 
  denote by $M[I]$  the $I$-torsion in $M$.
 Let    $I_{g_\alpha} \!\!\triangleleft  \  \TT$ be   the kernel of $\varphi_{g_\alpha}$, and set
 $$ S_1(Np,\chi)[g_\alpha] := S_1(Np,\chi)[I_{g_\alpha}], \qquad S_1(Np,\chi)[[g_\alpha]] := S_1(Np,\chi)[I_{g_\alpha}^2].
 $$
 Our main object of study is the  subspace 
$$  S_1^{(p)}(N,\chi)[[g_\alpha]] := S_1^{(p)}(N,\chi)[I_{g_\alpha}^2].$$
of the space of overconvergent $p$-adic modular forms of weight one, which is contained in 
the generalised eigenspace attached to $I_{g_\alpha}$. 
 An element of $S_1^{(p)}(N,\chi)[[g_\alpha]]$ is called an {\em overconvergent generalised
 eigenform} attached to $g_\alpha$, and it is said to be {\em classical} if it belongs to
 $S_1(Np,\chi)[[g_\alpha]]$. 
The theorem  of Bellaiche and Dimitrov stated in the opening paragraphs  of 
Part A  implies that the natural inclusion
$$ S_1(Np,\chi)[[g_\alpha]] \hookrightarrow S_1^{(p)}(N,\chi)[[g_\alpha]]$$
is an isomorphism, i.e., every overconvergent generalised eigenform is classical,
except possibly in the following cases:
\begin{itemize}
\item[(a)] $g$ is the theta series attached to  a finite order character of a real quadratic field
 in which the prime $p$ splits, or
\item[(b)] $g$ is {\em  irregular}  at $p$, i.e., $\alpha=\beta$. 
\end{itemize}
The study of $S_1^{(p)}(N,\chi)[[g_\alpha]]$ in    scenario (a)  was carried out in \cite{DLR3} when $\alpha\ne \beta$.
The main result of loc.cit.~is the description of a basis
 $(g_\alpha,g_\alpha^\flat)$ for
$S_1^{(p)}(N,\chi)[[g_\alpha]]$ which is {\em canonical up to scaling},
 and an expression for the fourier coefficients of the non-classical 
$g_\alpha^\flat$ (or rather, of their ratios)
in terms of $p$-adic  logarithms of certain  
algebraic numbers.

 Assume henceforth that $g$ is   not  regular at $p$, i.e., that $\alpha=\beta$. 
 In that case, the form $g$ admits a unique 
 $p$-stabilisation $g_\alpha=g_\beta$.
 The Hecke operators $T_\ell$ for $\ell\nmid Np$ and $U_\ell$ for $\ell\mid N$ act semisimply (i.e., as scalars) on 
 the two-dimensional vector space
 $$S_1(Np,\chi)[[g_\alpha]] = \Q_p g_\alpha \oplus \Q_p g', \qquad g'(q) := g(q^p),$$
  but the Hecke operator $U_p$ acts non-semisimply 
via the formulae
$$ U_p g_\alpha = \alpha g_\alpha, \qquad U_p g' = g_\alpha + \alpha g'.$$
 Because
 $$ (a_1(g_\alpha), a_p(g_\alpha)) = (1, \alpha), \qquad (a_1(g'), a_p(g')) = (0,1),$$
 the  classical subspace $S_1(Np,\chi)[[g_\alpha]]$
  has a natural linear complement in 
 $S_1^{(p)}(N,\chi)[[g_\alpha]]$, consisting of  the generalised eigenforms $\tilde g$ whose $q$-expansions satisfy
 \begin{equation}
 \label{eqn:normalised}
 a_1(\tilde g) = a_p(\tilde g) = 0.
 \end{equation}

 A modular form satisfying \eqref{eqn:normalised}
  is said to be {\em normalised}, and the space of
 normalised generalised eigenforms is denoted $S_1^{(p)}(N,\chi)[[g_\alpha]]_0$.
 The main goal of Part B is to study this space    and give an explicit description
 of its elements in terms of their
fourier expansions.
 The idoneous fourier coefficients will be  expressed  as  determinants of $2\times 2$
  matrices whose entries are  $p$-adic 
 logarithms of 
 algebraic numbers the number field  $H$ cut out by the projective Galois representation
 attached to $g$ (cf.~Theorems \ref{thm:exotic}, \ref{thm-CM} and \ref{thm-RMirreg}).

 \section{Generalised eigenspaces}
 \label{sec:dimension}
 
 We begin by recalling some of the notations that were already
 introduced in Part A. 
 Let 
 $$
\varrho_g: G_\Q \lra \Aut_{\Q_p}(V_g) \simeq \GL_2(\Q_p)
$$
be the odd, two-dimensional 
 Artin representation  
  associated to $g$ by Deligne and Serre
  (but viewed as having  $p$-adic rather than complex coefficients; as in Part A, we assume for simplicity that the image of $\varrho_g$ can be embedded in $\GL_2(\Q_p)$ and not just in $\GL_2(\bar\Q_p)$).
  
The four-dimensional $\Q_p$-vector space
  $ W_g:= \Ad(V_g):=\End(V_g)$  of  endomorphisms of $V_g$ is
 endowed with the  conjugation action  of $G_\Q$,
 $$\sigma \cdot M := \varrho_g(\sigma)  \circ M \circ \varrho_g(\sigma)^{-1},  \qquad \mbox{ for any } 
\sigma \in G_\Q, \quad M\in {W_g}. $$ 
Let $H$ be the field cut out by this Artin representation.
 The action  of $G_\Q$ on $W_g$ factors through a faithful action of the finite quotient $G:= \Gal(H/\Q)$. 
Let  $W_g^\circ := \Ad^0(V_g)$ denote the three-dimensional $G_\Q$-submodule of ${W_g}$  consisting of 
 trace zero endomorphisms.
 The exact sequence 
 $$ 0 \lra W_g^\circ \lra   W_g \lra \Q_p \lra 0$$
 of $G$-modules admits a canonical $G$-equivariant splitting 
 $$p:{W_g} \lra W_g^\circ, \qquad p(A) := A - 1/2 \cdot  {\rm Tr}(A).$$ 
  Because the action of $G_\Q$ on $V_g$ also factors through a finite quotient,
 the field $L \subset \Q_p$ generated by the traces of 
 $\varrho_g$ is a finite extension of $\Q$, and $\varrho_g$ maps the semisimple algebra
 $L[G_\Q]$ to a central simple algebra of rank $4$ over $L$. By eventually
  enlarging $L$, it can be assumed that
 $\varrho_g(L[G_\Q]) \simeq M_2(L)$, and therefore that $\varrho_g$ is realised on a two-dimensional
 $L$-vector space $V_g^L$ equipped with an identification $\iota: V_g^L\otimes_L \Q_p \lra V_g$. 
 The spaces 
 $$ {W}_g^L := \Ad(V_g^L), \qquad W_g^{\circ L} := \Ad^0(V_g^L)$$ 
 likewise correspond to  $G$-stable $L$-rational structures on ${W_g}$ and $W_g^\circ$ respectively, 
equipped with identifications
  $$\iota: {W}_g^L \otimes_L \Q_p \lra {W}_g, \qquad 
  \iota: W_g^{\circ L}\otimes_L \Q_p \lra W_g^\circ.$$
  
  The spaces  ${W}_g$ and $W_g^\circ$ (as well as 
  ${W}_g^L$ and $W_g^{\circ L}$)  are equipped with the Lie bracket $[\ ,\ ]$  and with a
 symetric non-degenerate pairing $\langle \ ,\ \rangle$ 
 defined by the usual rules
 $$ [A,B] := AB-BA, \qquad \langle A,B\rangle := {\rm Tr}(AB),$$
 which are compatible with the $G$-action in the sense that
 $$ [\sigma \cdot A, \sigma\cdot B] = \sigma\cdot [A,B], \qquad  \langle \sigma \cdot A, \sigma\cdot B\rangle 
 =\langle A,B\rangle,  \qquad \mbox{ for all } \sigma\in G.$$
 These operations can be combined to define a    $G$-invariant determinant function---i.e., a non-zero, alternating trilinear form---on $W_g^\circ$  and on $W_g^{\circ L}$ by setting
 $$ \det(A,B,C) := \langle[A,B],C\rangle.$$
  
The rule described in \eqref{correspondence} gives rise to natural identifications
$$
S_1(Np,\chi)[g_\alpha] \simeq \Hom(\TT/I_{g_\alpha},\Q_p), \quad S_1^{(p)}(N,\chi)[[g_\alpha]] \simeq \Hom(\TT/I_{g_\alpha}^2,\Q_p),
$$
and hence the dual of the short exact sequence
$$
0 \ra I_{g_\alpha}/I_{g_\alpha}^2 \ra  \TT/I_{g_\alpha}^2 \ra \TT/I_{g_\alpha} \ra 0
$$
can be identified with
$$ 0 \lra S_1(Np,\chi)[g_\alpha] \lra S_1^{(p)}(N,\chi)[[g_\alpha]] \lra S_1^{(p)}(N,\chi)[[g_\alpha]]_0 \lra 0.$$
In particular, one has the isomorphism
\begin{equation}\label{identification}
S_1^{(p)}(N,\chi)[[g_\alpha]]_0 \simeq \Hom(I_{g_\alpha}/I_{g_\alpha}^2,\Q_p).
\end{equation}
Let $\Q_p[\varepsilon] = \Q_p[x]/(x^2)$ denote the ring of dual numbers.
Given $ g^\flat \in S_1^{(p)}(N,\chi)[[g_\alpha]]_0$, the modular form
 $ {\tilde g} := g_\alpha + \varepsilon \cdot g^\flat$ is an 
  eigenform for $\TT$ with coefficients in
 $\Q_p[\varepsilon]$. Its associated Galois representation
$$
 \varrho_{\tilde g}: G_\Q \lra \GL_2(\Q_p[\varepsilon])
$$
satisfies
\begin{enumerate}
\item[(i)] $ \varrho_{\tilde g } = \varrho_g \pmod{\varepsilon}$ and $\det( \varrho_{\tilde g})  = \chi$,
\item[(ii)] for every prime number $\ell \nmid Np$, the trace of an arithmetic Frobenius 
$\tau_\ell$ at $\ell$ is 
\begin{equation}\label{trace}
\Tr( \varrho_{\tilde g}(\tau_\ell)) = a_\ell(g_\alpha) + \varepsilon\cdot a_\ell(g^\flat).
\end{equation}
\end{enumerate}

\begin{conjecture}\label{conj:main}
 Assume that $g$ is irregular at $p$. Then 
  the assignment $g^\flat\mapsto \varrho_{\tilde g}$ gives rise to a canonical isomorphism
 between $S_1^{(p)}(N,\chi)[[g_\alpha]]_0$ and the space $\mathrm{Def}^0(\varrho_g)$ of
 isomorphism  classes of  deformations of $\varrho_g$
   to the ring of dual numbers, with constant determinant. 
 \end{conjecture}
 We now derive some consequences of this conjecture.  
 
 \begin{proposition}\label{prop-S02d}
 \label{prop:partB}
 Assume Conjecture \ref{conj:main}. If $g$ is irregular at $p$, then the space
 $S_1^{(p)}(N,\chi)[[g_\alpha]]_0$ is two-dimensional over $\Q_p$.
 \end{proposition}
 \begin{proof}
Since  any $\tilde\varrho \in \mathrm{Def}^0(\varrho_g)$ has constant determinant, it may be written as
\begin{equation}
\label{c}
 \tilde\varrho = (1+   \varepsilon \cdot c)  \cdot \varrho_g \quad \mbox{ for some  } \quad c=c(\tilde\varrho): G_{\Q} \lra  W_g^\circ.
\end{equation}
The multiplicativity of  $ \tilde\varrho$ 
implies that the function $c$ is a
$1$-cocycle of $G_{\Q}$ with values in $W_g^\circ$, whose class in $H^1(\Q,W_g^\circ)$
(which shall be  denoted with the same symbol, by a slight abuse of notation) depends only
on the isomorphism class of $ \tilde\varrho$. The 
assignment $ \tilde\varrho \mapsto c( \tilde\varrho)$ realises an isomorphism  (cf.\,for instance \cite[\S 1.2]{Ma})
\begin{equation*}
\mathrm{Def}^0(\varrho_g) \lra H^1(\Q, W_g^\circ).
\end{equation*}
 Under Conjecture \ref{conj:main}, this yields an isomorphism
 \begin{equation}
\label{g-c}
 S_1^{(p)}(N,\chi)[[g_\alpha]]_0 \stackrel{\sim}{\lra} H^1(\Q, W_g^\circ), \quad g^\flat \mapsto c_{g^\flat}.
 \end{equation}
 
 The inflation-restriction sequence combined with global class field theory for $H$ 
 now gives rise to a series of identifications
\begin{eqnarray} 
\nonumber
 H^1(\Q,W_g^\circ)  & \stackrel{\res_H}{\lra} &  \hom(G_H, W_g^\circ)^G \\
 \nonumber
 &=& \hom_G\left( \frac{(\cO_H\otimes \Z_p)^\times}{\cO_H^\times\otimes\Z_p}, W_g^\circ\right)  \\
 \nonumber
 &=&  \hom_G\left( \frac{H_p}{U}, W_g^\circ\right)  \\ 
  \label{eqn:from-H1-to-hom}
 &=& \ker\left( \hom_G(H_p,W_g^\circ) \stackrel{\res_U}{\lra} \hom_G(U,W_g^\circ) \right),
 \end{eqnarray}
 where $U$ denotes the natural 
 image of $\cO_H^\times\otimes \Z_p$ in $H_p := H\otimes \Q_p$ under the 
 $p$-adic logarithm
 map 
 $$ \log_p: H_p^\times \lra H_p.$$ 
 As representations for $G$, the space  $H_p$ is isomorphic to the regular representation
 $$ H_p \simeq \Ind_1^G \Q_p,  $$
  while  $U$, by the Dirichlet unit theorem, 
    is induced from the trivial representation of the subroup
  $G_\infty\subset G$ generated by a complex conjugation:
$$ U\simeq \Ind_{G_\infty}^G \Q_p.$$
 Complex conjugation acts on $W_g^\circ$ with eigenvalues $1$, $-1$ and$-1$, and hence
 by  Frobenius reciprocity, 
 \begin{equation}
 \label{eqn:dirichlet-unit-thm}
  \dim_{\Q_p} \hom_G(H_p,W_g^\circ) = 3, \qquad \dim_{\Q_p}\hom_G(U,W_g^\circ)=1.
  \end{equation}
 It follows from  
  \eqref{eqn:from-H1-to-hom}
 that $H^1(\Q,W_g^\circ)$ is two-dimensional over $\Q_p$.
 Proposition \ref{prop:partB}   follows.
 \end{proof}

 For any $\ell\nmid Np$, the $\ell$-th fourier coefficient of $g^\flat$ is given in terms of the associated cocycle 
 $c_{g^\flat}$ by the rule
 \begin{equation}
 \label{eqn:formula-for-al}
  a_\ell(g^\flat) =  
  {\rm Tr}(c_{g^\flat}(\sigma_\lambda) \varrho_g(\sigma_\lambda))
  \end{equation}
  where $\lambda|\ell$ is any prime above $\ell$ and $\sigma_\lambda$ denotes the arithmetic Frobenius associated to it. Note that the right-hand side of \eqref{eqn:formula-for-al} does not depend on the choice of $\lambda$.

 Our next goal is to  parametrise the  elements 
of \eqref{eqn:from-H1-to-hom}  explicitly, and then to derive concrete formulae for the 
fourier expansions of
the associated modular forms in  $S_1^{(p)}(N,\chi)[[g_\alpha]]_0$ via  \eqref{g-c} 
and \eqref{eqn:formula-for-al}.
After treating the general case in  Section \ref{sec:exotic},
Sections   \ref{sec:cm} and \ref{sec:rm}  focus on the special features of the scenarios where 
$W_g^\circ$ is reducible, i.e.,
 \begin{enumerate}
 \item[(i)] the CM case where $V_g$ is induced from a character of an imaginary quadratic field;
 \item[(ii)] the RM case where $V_g$ is induced from a character of a real quadratic field.
  \end{enumerate}

\section{The general case}
\label{sec:exotic}
The
  Galois representation $W_g^\circ$ is irreducible if and only if 
 $G:= \Gal(H/\Q)$ is  isomorphic to   $A_4$, $S_4$, or $A_5$. Otherwise, the representation $\varrho_g$ has dihedral projective image and $G$ is isomorphic to a dihedral group.
 
The irregularity assumption implies that the prime $p$ splits completely in $H$, and  
$H$ can therefore be viewed as a subfield of $\Q_p$ after fixing an embedding
 $H\hookrightarrow \Q_p$ once and for all. This amounts to choosing a prime $\wp$ of $H$ above $p$.
 Let $\log_{\wp}: H_p^\times \lra \Q_p$ denote the associated ${\wp}$-adic logarithm map, which factors through 
 $\log_p$.

 The Dirichlet unit theorem implies (via  the second equation in \eqref{eqn:dirichlet-unit-thm})
 that
 $$ \dim_L (\cO_H^\times \otimes W_g^{\circ L})^G = 1.$$ 
 In particular, for all $\fu\in \cO_H^\times $ and all $w\in W_g^{\circ L}$, the element
 \begin{equation}
 \label{eqn:xiuw}
  \xi(\fu,w) := \frac{1}{\# G} \times \sum_{\sigma\in G} (\sigma \fu) \otimes (\sigma\cdot w) \in (\cO_H^\times \otimes W_g^{\circ L})^G
  \end{equation}
 only depends on the choices of $\fu$ and $w$ up to scaling by a (possibly zero) factor in $L$.
 As $\fu$ varies over $\cO_H^\times$ and $w$ over $W_g^{\circ L}$,  
 the elements
 \begin{equation}
 \label{eqn:xiuw-p}
  \xi_{\wp}(\fu,w) := ( \log_{\wp}\otimes {\rm id}) \xi(\fu,w) = 
   \frac{1}{\# G} \times \sum_{\sigma\in G} \log_{\wp}(\sigma \fu)  \cdot (\sigma\cdot w) \in  W_g^\circ 
  \end{equation}
 therefore
  lie in a one-dimensional $L$-vector subspace of $W_g^\circ$. 
  Choose a generator 
$w(1)$ for this space.
The coordinates of $w(1)$ relative to a basis  $(e_1, e_2,e_3)$
 for $W_g^{\circ L}$ are $\wp$-adic logarithms of units in $\cO_H$, namely, we can write
\begin{equation}
\label{eqn:coord-u}
 w(1) = \log_{\wp}(\fu_1) e_1 + \log_{\wp}(\fu_2)  e_2 + \log_{\wp}(\fu_3) e_3,
 \end{equation}
 for appropriate $\fu_i\in (\cO_H^\times)\otimes_{\Z} L$.
  
Let $\ell\nmid N p$ be a rational prime.  
For   any prime 
$\lambda$ of $H$ above $\ell$,  
let $\tilde\fu_\lambda$ be a generator of the principal ideal $\lambda^h$, 
where $h$ is the class number of $H$, and set 
$$ \fu_\lambda := {\tilde\fu}_\lambda\otimes h^{-1} \in (\cO_H[1/\ell]^\times)\otimes_\Z L.$$
 Let 
\begin{equation}
\label{eqn:def-w-lambda}
 \tilde w_\lambda:=   \varrho_g(\sigma_\lambda) \in {W_g^L}, \qquad
w_\lambda := p({\tilde w_\lambda}) \in W_g^{\circ L}
\end{equation}
be the  endomorphisms of $V_g$
arising from the 
image of  $\sigma_\lambda$ under $\varrho_g$. 
The element $\fu_\lambda$ is well-defined up to multiplication by elements of $\cO_H^\times$,
and hence the elements 
\begin{eqnarray}
\label{eqn:xiuw-bis}
  \xi(\fu_\lambda,w_\lambda) &:=&   \frac{1}{\# G} \times \sum_{\sigma\in G} (\sigma \fu_\lambda) \otimes (\sigma\cdot w_\lambda) \in (\cO_H[1/\ell]^\times \otimes W_g^{\circ L})^G, \\
  \nonumber
 w(\ell) =  \xi_{\wp}(\fu_\lambda,w_\lambda) & := & 
 \frac{1}{\# G} \times \sum_{\sigma\in G} \log_{\wp}(\sigma \fu_\lambda)\cdot (\sigma\cdot w_\lambda) \in   W_g^\circ
 \end{eqnarray}
 are   defined up to  translation by  elements of the one-dimensional $L$-vector spaces
 $(\cO_H^\times\otimes W_g^{\circ L})^G$ and $L \cdot w(1)$ respectively.
 Furthermore, the image of $w(\ell)$ in the quotient $W_g^\circ/(L \cdot w(1))$ does not depend
 on the choice of the prime $\lambda$ of $H$ above $\ell$ that was made to define it. 
The  Lie bracket
 $$ \fw(\ell) := [w(1),w(\ell)] \in W_g^\circ$$
 is thus independent of  the choices that were made in  defining $w(\ell)$.
  
  \begin{remark}
The coordinates of $w(\ell)$ relative to a basis  $(e_1, e_2,e_3)$
 for $W_g^{\circ L}$ are $\wp$-adic logarithms of $\ell$-units in $H$, i.e., one can write
\begin{equation}
\label{eqn:coord-v}
 w(\ell) = \log_{\wp}(\fv_1) e_1 + \log_{\wp}(\fv_2)e_2 + \log_{\wp}(\fv_3) e_3, 
 \end{equation}
 with $\fv_i\in (\cO_H[1/\ell]^\times)_L$  for $i =1,2,3$.
 A direct computation shows that
$$  
 \dim_{\Q_p} (W_g^\circ)^{\sigma_\lambda=1} = \left\{ \begin{array}{cl}  1 & \mbox{ if  $g$ is regular at $\ell$};   \\
                 3 & \mbox{ if $g$ is irregular at $\ell$}.   \end{array}
                 \right. 
        $$
It follows that for all regular primes $\ell$,
$$  \dim_L(\cO_H[1/\ell]^\times \otimes W_g^{\circ L})^G = 2, $$
and therefore that the element $\xi(\fv,w)$ attached to any pair 
$(\fv,w) \in \cO[1/\ell]^\times \times W_g^{\circ L }$ as in 
\eqref{eqn:xiuw-bis} 
is well-defined up to scaling by $L$ and up to translation by elements
of the one-dimensional space $(\cO_H^\times\otimes W_g^{\circ L})^G$.  
In particular, the associated vector $\fw(\ell)$ lies in a canonical one-dimensional subspace
of $W_g^\circ$, namely, the orthogonal complement in $W_g^{\circ}$ of 
$$(\log_{\wp}\otimes {\rm Id})(\cO_H[1/\ell]^\times\otimes W_g^{\circ})^G \subset W_g^{\circ}.$$  
If
 the basis $(e_1e_2,e_3)$ for $W_g^{\circ L}$  in 
\eqref{eqn:coord-u} and \eqref{eqn:coord-v} is taken to be the standard basis
$$ 
e_1 = \left(\begin{array}{cc}  1 & 0 \\ 0 & -1 \end{array}\right), \qquad
e_2 = \left(\begin{array}{cc}  0 & 1 \\ 0 & 0 \end{array}\right), \qquad
e_3 = \left(\begin{array}{cc}  0 & 0 \\ 1 & 0 \end{array}\right), 
$$
then 
$$ \fw(\ell)  \   \ =\   \ \det\left(\begin{array}{cc} \fu_2 & \fu_3 \\ \fv_2 & \fv_3 \end{array}\right) \cdot e_1 
 \ + \ 2 \det \left(\begin{array}{cc} \fu_1 & \fu_2 \\ \fv_1 & \fv_2 \end{array}\right) \cdot e_2  
\ - \ 2 \det \left(\begin{array}{cc} \fu_1 & \fu_3 \\ \fv_1 & \fv_3 \end{array}\right) \cdot e_3.$$
  \end{remark}
  
\begin{remark}
Observe that if the prime $\ell$ is irregular for $g$, the vector $\tilde w_\lambda$ is a scalar 
endomorphism in ${W_g}$ and hence $w_\lambda =  w(\ell) = \fw(\ell) =0$. 
\end{remark}

Our main result is 
 \begin{theorem}
 \label{thm:exotic}
 Assume Conjecture \ref{conj:main}.
 For all $w\in W_g^\circ$,  there exists an overconvergent generalised eigenform $g_w^\flat \in 
 S_1^{(p)}(N,\chi)[[g_\alpha]]_0$  satisfying
 $$ a_\ell(g_w^\flat) =   \langle  w, \fw(\ell)  \rangle = \det(w,w(1),w(\ell)),$$
 for all primes $\ell\nmid Np$. 
 The assignment $w\mapsto g_w^\flat$ induces 
 an isomorphism between $W_g^\circ/U$ and $S_1^{(p)}(N,\chi)[[g_\alpha]]_0$.
  \end{theorem}
\begin{proof} 
The semi-local field
 $H_p  =  H\otimes_\Q \Q_p= \oplus_{\wp|p} \Q_p$ is naturally identified with the set of vectors
$h = (h_{\wp})_{\wp|p}$ with entries $h_{\wp}\in\Q_p$, indexed by the primes of $H$ above $p$.
The function which to $ w\in W_g^\circ$ associates the linear transformation
$$ \tilde\varphi_{w}: H_p \lra W_g^\circ, \qquad \tilde\varphi_{w}(h) = 
 \frac{1}{\# G} \times \sum_{\sigma\in G} (\sigma^{-1}h)_{\wp} \cdot (\sigma  \cdot {w}) $$
identifies $W_g^\circ$ with $\hom_G(H_p,W_g^\circ)$. 
 The linear function $\tilde\varphi_{w}$ is trivial on $U:= \log_p(\cO_H^\times)\subset H_p$ 
 if and only if, for all $\fu\in \cO_H^\times$ and all $w'\in W_g^\circ$,
 $$  \langle \tilde\varphi_{w}(\log_p(\fu)), w' \rangle = 0.$$
 But
 \begin{eqnarray*}
 \langle \tilde\varphi_{w}(\log_p(\fu)), w'\rangle &=&  \frac{1}{\# G} \times \left\langle  \sum_{\sigma\in G} \log_{\wp}(\sigma^{-1}(\fu)) \cdot (\sigma\cdot {w}),  \ w' \right\rangle \\ 
 &=&     \frac{1}{\# G} \times\sum_{\sigma\in G} \log_{\wp}(\sigma^{-1}(\fu)) \cdot  \langle \sigma\cdot {w}, w' \rangle \\ 
 &=&    \frac{1}{\# G} \times \sum_{\sigma\in G} \log_{\wp}(\sigma^{-1}(\fu)) \cdot  \langle   {w}, \sigma^{-1} \cdot w' \rangle \\
 &=& \langle {w}, \xi_{\wp}(\fu,w') \rangle,
 \end{eqnarray*}
 and hence $\tilde \varphi_{w}$  is trivial
  on $U=\log_p(\cO_H^\times)$ if and only if 
 $w$ is orthogonal in $W_g^\circ$ to the line spanned by 
 $w(1)$. It follows that the $G$-equivariant linear function 
 $$ \varphi_w := \tilde\varphi_{[w,w(1)]}: H_p \lra W_g^\circ$$
 factors through $H_p/U$. 
The assignment $w\mapsto \varphi_w$ identifies
 $W_g^\circ/ (L\cdot w(1))$ with $\hom_G( H_p/U, W_g^\circ)$, and gives an explicit description of the 
 latter space.
 
 Let $\tilde g_w = g + \varepsilon g_w^\flat$ be the eigenform with coefficients in 
 $\Q_p[\varepsilon]$ which is attached to the cocycle 
 $\varphi_w\in \hom_G(H_p/U,W_g^\circ) = H^1(\Q,W_g^\circ)$. 
Equation  \eqref{eqn:formula-for-al}  with $g^\flat = g_w^\flat$ (and hence $c_{g^\flat} = \varphi_w$) 
combined with \eqref{eqn:def-w-lambda} shows
 that the $\ell$-th  the fourier coefficient of $g_w^\flat$ 
at a prime $\ell \nmid Np$ is equal to 
\begin{equation}\label{al}
a_\ell(g_w^\flat) =      
 {\rm Tr}(\varphi_w(\sigma_\lambda) \cdot \varrho_g(\sigma_\lambda))  
=  \langle \varphi_w(\sigma_\lambda), \tilde w_\lambda\rangle 
=  \langle \varphi_w(\sigma_\lambda), w_\lambda\rangle.
\end{equation}
Class field theory for $H$ implies that
$$ \varphi_w(\sigma_\lambda) =  \frac{1}{\# G} \times \sum_{\sigma\in G} \log_{\wp}(\sigma^{-1} \fu_\lambda) \cdot 
\sigma\cdot[w,w(1)].$$  
Hence
\begin{eqnarray*}
 a_\ell(g_w^\flat) &=&     \frac{1}{\# G} \times \left\langle \sum_{\sigma\in G} \log_{\wp}(\sigma^{-1} \fu_\lambda) \cdot 
\sigma\cdot[w,w(1)],   \   w_\lambda  \right\rangle \\
&=&  \frac{1}{\# G} \times \sum_{\sigma\in G} \log_{\wp}(\sigma^{-1} \fu_\lambda) \cdot  \langle \sigma\cdot[w,w(1)],  \ w_\lambda \rangle  \\ 
&=& \frac{1}{\# G} \times \sum_{\sigma\in G} \log_{\wp}(\sigma^{-1} \fu_\lambda) \cdot  \langle [w,w(1)],   \ \sigma^{-1}\cdot w_\lambda  \rangle  \\ 
&=&  \langle [w,w(1)], w(\ell)   \rangle   = \det(w,w(1),w(\ell)) = \langle w, \fw(\ell)\rangle.
\end{eqnarray*}
The theorem follows.
\end{proof}
If $w$ in a vector in $W_g^{\circ L}$, Theorem \ref{thm:exotic}
 shows that the associated overconvergent generalised eigenform $g_w^\flat$ has
fourier coefficients which are   $L$-rational linear combinations of 
determinants of $2\times 2$  matrices whose entries are the ${\wp}$-adic logarithms of algebraic numbers 
in $H$. 
In the CM and RM cases to be discussed below, the representation $W_g^{\circ}$ is reducible and decomposes further into non-trivial irreducible representations.
In that case  the choice of an  $L$-basis for $W_g^{\circ L}$ which is compatible with this decomposition leads to  canonical elements of 
$S_1^{(p)}(N,\chi)[[g_\alpha]]_0$ which can sometimes be re-scaled so 
that their fourier expansions  admit even simpler 
expressions, as will be described in the next two sections.




 \section{CM forms}
 \label{sec:cm}
Assume that $g$ is the theta series attached to a character of a quadratic imaginary field
$K$, i.e., that
$$V_g^L = \Ind_K^\Q \psi_g,$$
where $\psi_g: \Gal(\bar K/K) \lra L^\times$ is a finite order character.
Let $\psi_g'$ denote the character deduced from $\psi_g$ by composing it 
with the involution in $\Gal(K/\Q)$.
The  irreducibility assumption on $V_g^L$ implies that  the characters $\psi_g$ and $\psi_g'$ are distinct,
and therefore the representations $V_g^L$  and $V_g$ decompose canonically as a direct
sum of two $G_K$-stable one-dimensional subspaces
$$ V_g^{L} = \cL_{\psi_g}^L \oplus \cL_{\psi_g'}^L, \qquad
 V_g  = \cL_{\psi_g}  \oplus \cL_{\psi_g'}$$
 on which $G_K$ acts via the characters $\psi_g$ and $\psi_g'$ respectively.
The representations ${W}_g^L$ and 
${W_g}$  also decompose  as   direct sums of four $G_K$-stable lines
\begin{eqnarray*}
{ W_g^L} &=& \left(\hom(\cL_{\psi_g}^L,\cL_{\psi_g}^L) \oplus 
\hom(\cL_{\psi_g'}^L,\cL_{\psi_g'}^L) \right)  \ \oplus  \ 
\left(\hom(\cL_{\psi_g'}^L,\cL_{\psi_g}^L)   \oplus
\hom(\cL_{\psi_g}^L,\cL_{\psi_g'}^L)  \right),    \\
{  W_g} &=& \left(\hom(\cL_{\psi_g},\cL_{\psi_g}) \oplus 
\hom(\cL_{\psi_g'},\cL_{\psi_g'}) \right)  \ \oplus  \ 
\left(\hom(\cL_{\psi_g'},\cL_{\psi_g})   \oplus
\hom(\cL_{\psi_g},\cL_{\psi_g'})  \right).
\end{eqnarray*}
The direct summands  in parentheses are also stable under $G_\Q$ and are isomorphic to the induced representations
$\Ind_K^\Q 1$ and $\Ind_K^\Q\psi$ respectively, where 
 $\psi :=\psi_g/\psi_g'$, is the {\em ring class character} of $K$ associated to $\psi_g$. 
  It follows that 
$$ W_g^{\circ L}  = L(\chi_K)  \oplus  Y_g^{L}, \qquad 
W_g^\circ  = \Q_p(\chi_K)  \oplus  Y_g, \qquad  Y_g^L :=  \Ind_K^\Q \psi, \quad 
Y_g := Y_g^L \otimes_L \Q_p.$$
 It will be convenient to choose 
 a basis $(e_1,e_2)\in \cL_{\psi_g}^L \times \cL_{\psi_g'}^L$ for $V_g^{L}$,  
 and to denote by $e_{11}, e_{12}, e_{21}, e_{22}$ the resulting basis of 
 ${ W}_g^L$, where $e_{ij}$ is the elementary matrix whose $(i',j')$-entry is 
 $\delta{i=i'} \delta_{j=j'}$.
 Relative to the  identification of  $W_g^{\circ L}$ with the space of $2\times 2$ matrices of trace zero with entries in $L$ via this basis, the  representation
 $L(\chi_K) = L\cdot(e_{11}-e_{22}) $ is identified with the space of diagonal matrices of trace $0$, while $Y_g^{L} = L\cdot e_{12} \oplus L\cdot e_{21}$ is identified with the 
 space of off-diagonal matrices in $M_2(L)$.  
 Fix an element $\tau \in G_\Q = G_K$ once and for all.
By eventually re-scaling $e_1$ and $e_2$, it can (and shall, henceforth) be assumed  that  
 $\varrho_g(\tau)$ is  represented by  the matrix
 $ \left( \begin{array}{cc} 0 & t  \\ t  & 0 \end{array}\right)$ 
 in this basis, where $- t^2 := \chi(\tau)$. 
 
 Let $Z:= \Gal(H/K)$ be the maximal abelian normal subgroup 
of the dihedral group $G = \Gal(H/\Q)$. Note that every element in $G-Z$ (such as  the image of $\tau$ 
  in $G$) is an involution, and that $Z$ operates transitively on $G-Z$ by either left or right multiplication.
 
The field 
  $H$ through which $W_g^\circ$ factors  is  the  ring class field of $K$  attached to the character $\psi$.
 The group $\cO_H^\times \otimes \Q$ of units of $H$ is isomorphic to the regular representation of $Z$ minus the trivial representation, and a finite index subgroup of
  $\cO_H^\times$ can be constructed explicity from the elliptic units arising in the theory of complex multiplication. 
  Let 
  $$ e_{\psi} := \frac{1}{\# Z} \sum_{\sigma\in Z} \psi^{-1}(\sigma) \sigma$$
   be the idempotent in the group ring of $Z$ giving rise to the projection onto the  
   $\psi$-isotypic component for the action of $Z$. 
   Choose a unit $\fu\in \cO_H^\times$ and let 
 $$\fu_\psi : = e_\psi \fu, 
\qquad  \tau\fu_{\psi} =   e_{\psi'} (\tau \fu) $$
  be  elements of $\cO_H^\times \otimes L $ on which $Z$ acts via the characters 
  $\psi$ and   $\psi' = \psi^{-1}$ respectively.
With these choices, we can let 
\begin{equation}
\label{eqn:unit-imag}
 w(1) = \left( \begin{array}{cc}  0   & \log_{\wp}(\fu_\psi) \\
\log_{\wp}(\tau\fu_\psi) & 0 \end{array}\right).
\end{equation}

The description of the canonical vectors $w(\ell), \fw(\ell) \in W_g^\circ$ attached to a rational prime $\ell\nmid Np$ depends in an essential way on whether   $\ell$ is split or inert in $K/\Q$.

If $\ell = \lambda\lambda'$ is split in $K$ and $\ell$ is regular for $g$, i.e.,
$\psi_g(\sigma_{\lambda}) \ne \psi_g({\lambda'})$,  then the natural map
  $$(\cO_K[1/\ell]^\times\otimes  W_g^{\circ L})^G \subset \left(\frac{\cO_H[1/\ell]^\times}{\cO_H^\times} \otimes  W_g^{\circ L}\right)^G $$ is  an isomorphism of $L$-vector spaces.

 Let ${\tilde \fu}_\lambda$ be a generator of $\lambda^h$ where $h$ is the class number of $K$, and set
$$ \fu_\lambda :=   {\tilde \fu}_\lambda \otimes h^{-1}.$$
  Since
  $$
   {\tilde w}_\lambda = \left(\begin{array}{cc} \psi_g(\lambda)  &  0 \\ 
0 &   \psi_g({\lambda'})  \end{array} \right), \qquad
  { w}_\lambda = \frac{ \psi_g(\lambda) - \psi_g({\lambda'})}{2} \times \left(\begin{array}{cc} 1  &  0 \\ 
0 & -1 \end{array} \right), $$
a direct calculation shows that 
$$ w(\ell) =   \log_{\wp}(\fu_\lambda/\fu_\lambda') \times  \frac{(\psi_g(\lambda) - \psi_g({\lambda'}))}{2} \times
\left(\begin{array}{cc} 1  &  0 \\ 
0 &  -1 \end{array} \right).$$
It follows that
\begin{equation}
\label{eqn:fwl-quad-split}
\fw(\ell) =   \log_{\wp}(\fu_\lambda/\fu_\lambda') \times (\psi_g(\lambda) - \psi_g(\lambda')) \times \left( \begin{array}{cc}  0   & -\log_{\wp}(\fu_\psi) \\
\log_{\wp}(\fu_\psi') & 0 \end{array}\right).
\end{equation}

If $\ell $ is inert in $K$ then $\ell$ is always regular for $g$ since  $\varrho_g(\tau_\ell)$ has
trace $0$ and hence has distinct eigenvalues. The prime $\ell$ splits completely in $H/K$, and hence the group $(\cO_H[1/\ell]^\times)\otimes L$ is isomorphic to two copies of the regular representation of 
$Z$ minus a  trivial representation.  
The choice of a prime $\lambda$ of $H_g$ above $\ell$ determines 
a matrix (and not just a conjugacy class) 
$$  {\tilde w}_\lambda = w_\lambda = \varrho(\sigma_\lambda) = \left( \begin{array}{cc} 0 & b_\lambda \\ c_\lambda & 0 \end{array}\right)$$
with entries in $L$. 
Let
${  \fu}_\lambda$ be an element of $(\cO_H[1/\ell]^\times/\cO_H^\times) \otimes L $ whose prime
factorisation is given by 
\begin{equation}
\label{eqn:def-u-lambda-cm}
 (\fu_\lambda) =    b_\lambda \lambda + c_\lambda (\tau\lambda).
 \end{equation}
This $\ell$-unit is only well defined by \eqref{eqn:def-u-lambda-cm}
 up to translation by $\cO_H^\times\otimes L$,  and the defining equation
 \eqref{eqn:def-u-lambda-cm} of course depends crucially on the choice of 
 the prime $\lambda$ above $\ell$. However, the $\psi$-isotypic projection
 \begin{equation}
 \fu_\psi(\ell) := e_\psi \fu_\lambda
 \end{equation}
 is independent of this choice. 
 A direct calculation shows that 
  $$ w(\ell) = \frac{1}{2}    \times 
\left(\begin{array}{cc} 0   &    \log_{\wp}( \fu_\psi(\ell))    
\\ 
  \log_{\wp}( \tau \fu_{\psi}(\ell))  
 &  0 \end{array} \right).$$
It follows that
\begin{equation}
\label{eqn:fwl-quad-inert}
\fw(\ell) =   \left( \begin{array}{cc} R_\psi(\ell) & 0 \\ 0 & - R_\psi(\ell) \end{array} \right),
 \end{equation}
where 
$$ R_\psi(\ell) = \det\left(  \begin{array}{cc}  
 \log_{\wp}(\fu_\psi)  &    \log_{\wp}(\tau\fu_\psi)     \\
  \log_{\wp}(\fu_\psi(\ell)) &  \log_{\wp}(\tau\fu_\psi(\ell))  
  \end{array} \right) $$
is an $\ell$-unit regulator attached to $\psi$, 
 which is  independent of the choice of prime $\lambda$ of $H$ above 
$\ell$. 
The function $\ell\mapsto R_\psi(\ell)$ does depend on the choice of the  
unit $\fu$, but only up to scaling by $L^\times$.
\begin{theorem}\label{thm-CM}
The space  $S_1^{(p)}(N,\chi)[[g_\alpha]]_0$ 
has a canonical basis $(g_1^\flat, g_2^\flat)$ which is characterised by
the properties:
\begin{enumerate}  
\item[(i)]
The fourier coefficients  $a_\ell(g_1^\flat)$ are $0$ for  all primes $\ell\nmid Np$ that are inert in $K$.
 If $\ell = \lambda\lambda'$ is
split in $K$, then 
$$a_\ell(g_1^\flat) =  (\psi_g(\lambda) - \psi_g(\lambda'))  \times
  \log_{\wp}(\fu_\lambda/\fu_\lambda')   $$ 
is a simple algebraic multiple of the $p$-adic logarithm of the fundamental $\ell$-unit of norm 
$1$ in $K$.
\item[(ii)]
The fourier coefficients of $g_2^\flat$ are $0$ at all the primes $\ell\nmid Np$ that are split in $K$. 
If $\ell$ is inert in $K$ ,  then
$$ a_\ell(g_s^\flat) = R_\psi(\ell).$$
\end{enumerate}
\end{theorem}
\begin{proof}
This follows directly from the calculation of the matrices $\fw(\ell)$ in 
\eqref{eqn:fwl-quad-split}
and 
\eqref{eqn:fwl-quad-inert} in light of 
Theorem \ref{thm:exotic}.
\end{proof} 

\begin{example}
Let $\chi$ be the quadratic character of conductor $59$. The space $S(59,\chi)$ is one dimensional
and spanned by the theta series
\[ g = q - q^3 + q^4 - q^5 - q^7 - q^{12} + q^{15} + q^{16} + 2 q^{17} - \cdots .\]
Here $K = \Q(\sqrt{-59})$ and the ring class field attached to $\psim$ is
\[ H = K(\alpha) \mbox{ where } \alpha^3 - 3 \alpha + 46 \sqrt{-59} = 0.\]
The inert primes $\ell$ in $K$ are $2,3,13,23,\cdots$ and 
the unit and first few $\ell$-units are
\[
\begin{array}{rclrcl}
u  & = & \frac{1}{612} \left( 13 \alpha^2 - 7 \sqrt{-59} \alpha - 26 \right),  & u_2  & = & \frac{1}{612} \left( -5 \alpha^2 - 13 \sqrt{-59} \alpha - 194 \right)\\
\\
u_{11}   & =  & \frac{1}{306} \left( 5 \alpha^2 + 13 \sqrt{-59} \alpha - 112 \right), & u_{13}  & = &  \frac{1}{612} \left( 13 \alpha^2 - 7 \sqrt{-59} \alpha - 1250 \right)\\
\\
u_{23} &  = &  \frac{1}{204} \left(- \alpha^2 + 11 \sqrt{-59} \alpha + 138 \right). & & &
\\ 
\end{array}
\]
Let $p = 17$, an irregular prime for $g$. We computed a basis of $q$-expansions for the generalised eigenspace modulo $p^{20}$ and $q^{30,000}$. One observes
that it contains the classical space spanned by the forms $g_\alpha(q)$ and $g(q^p)$ and in addition a complementary space
of dimension two. This space is canonically spanned by two normalised generalised eigenforms 
\[ \tilde{g}_1^\flat = q^3 + \cdots + 0 \cdot q^{p} + \cdots \quad \mbox{and} \quad
\tilde{g}_2^\flat = q^2 + 0 \cdot q^3 + \cdots +  \cdots + 0 \cdot q^{p} + \cdots.\]
Note that the natural scaling of the forms output by our algorithm
is with leading Fourier coefficients equal to $1$. By Theorem \ref{thm-CM} one expects that
for $\ell$ inert in $K$, or $\ell$ split in $K$ but irregular, we have $a_\ell(\tilde{g}_1^\flat) = 0$; and
for $\ell$ split in $K$ we have that 
\[ a_\ell(\tilde{g}_1^\flat) = \frac{\log_p (u_\ell)}{\log_p(u_3)}\]
where $u_\ell$ is a fundamental $\ell$-unit in $K$ (the logarithm of this is well-defined up to sign). 
We checked this to $20$-digits of $17$-adic precision for primes $\ell < 1000$. Further, one expects that 
\[  a_\ell(\tilde{g}_2^\flat) = \frac{R_\psim(\ell)}{R_\psim(2)} \mbox{ for $\ell$ inert in $K$, and } a_\ell(\tilde{g}_2^\flat) = 0 \mbox{ for $\ell$ split in $K$.} \]
We checked this for all split primes $\ell < 30,000$ and for the inert primes $\ell = 2,3,11$ and $23$, constructing
$R_\psim(\ell)$ using the unit $u$ and $\ell$-unit $u_\ell$ above.
\end{example}

\section{RM forms}
\label{sec:rm}
We now turn to the RM setting where 
$F$ is a real quadratic  field and 
$$V_g = \Ind_F^\Q \psi_g,$$
where $\psi_g: \Gal(\bar F/F) \lra L^\times$ is a finite order character of mixed signature.
Letting $\psi_g'$ denote the character deduced from $\psi_g$ by composing it 
with the involution in $\Gal(F/\Q)$, the ratio
 $\psi :=\psi_g/\psi_g'$   is a totally odd $L$-valued  ring class character of $F$. 
 
 As before,  let $H$ denote the ring class field of $F$ which is fixed by the kernel of $\psi$, and set
 $Z:= \Gal(H/F)$ and $G:= \Gal(H/\Q)$.
 Just as in the previous section,
 $$ W_g^\circ = \chi_K \oplus  Y_g, \qquad  Y_g :=  \Ind_K^\Q \psi,$$
 and we can set 
 $$ w(1) = \left( \begin{array}{cc} 
 \log_{\wp}(\fu_F) & 0 \\ 
 0 & - \log_{\wp}(\fu_F) \end{array} \right),$$
 where $\fu_F$ is a fundamental unit of   $F$.

If $\ell$ is split in $K/\Q$, it is easy to see that the vector $w(\ell)$ is proportional
to $w(1)$, and hence that
\begin{equation}
\label{eqn:w-ell-real-split}
\fw(\ell) = 0.
\end{equation}
If $\ell$ is inert in $K$,  let  $U_g$ and $U_g(\ell)$ denote the subspaces
$(\cO_H^\times \otimes Y_g)^{G_\Q}$ and $(\cO_H[1/\ell]^\times \otimes Y_g)^{G_\Q}$.  
The
 the dimensions of these spaces are $0$ and $1$ respectively. 
Choose a prime $\lambda$  of $H$ above $\ell$, and let 
$\fu_\lambda$ and $\fu_\psi(\ell)$ be the elements of 
$\cO_H[1/\ell]^\times$ determined by the relations
$$(\fu_\lambda) =  b_\lambda \lambda + c_\lambda \tau\lambda, \qquad \fu_\psi(\ell) = e_\psi(\fu_\lambda),   \qquad \fu_\psi'(\ell) = \tau \fu_\psi(\ell),$$
where 
$$ \varrho_g(\sigma_\lambda)  = \left(\begin{array}{cc} 0 & b_\lambda \\ c_\lambda & 0 \end{array} 
\right).$$
The $\wp$-adic  logarithms
$$ \log_{\wp}(\fu_\psi(\ell)),  \qquad 
\log_{\wp}(\fu_\psi'(\ell))  $$
are well-defined  invariants of $\ell$ and $\varrho$ which do not depend on the choice of a prime 
$\lambda$ lying above $\ell$, and 
$$ w(\ell) = \frac{1}{2} \times \left(\begin{array}{cc}
0 & \log_{\wp}(\fu_\psi(\ell))  \\ 
\log_{\wp}(\fu_\psi'(\ell)) & 0  \end{array} \right).$$
It follows that
\begin{equation}
\label{eqn:fw-real-inert}
 \fw(\ell) = \log_{\wp}(\fu_F) \times 
 \left(\begin{array}{cc}
0 & \log_{\wp}(\fu_\psi(\ell))  \\ 
-\log_{\wp}(\fu_\psi'(\ell)) & 0  \end{array} \right).
\end{equation}
\begin{theorem}
\label{thm-RMirreg}
The space  $S_1^{(p)}(N,\chi)[[g_\alpha]]_0$ has a canonical basis $(g_1^\flat, g_2^\flat)$ which is characterised by
the properties:
\begin{enumerate}  
\item[(i)]
The fourier coefficients of $g_1^\flat$ and $g_2^\flat$ are $0$ at all primes $\ell\nmid Np$ that are split in $F$. 
\item[(ii)]
 If $\ell$ is inert in $F$,
then
$$ a_\ell(g_1^\flat) = \log_{\wp}(\fu_\psi(\ell)), 
\qquad  
 a_\ell(g_2^\flat) = \log_{\wp}(\fu_\psi'(\ell)).
$$
\end{enumerate}
\end{theorem}
\begin{proof} This follows directly from Theorem \ref{thm:exotic} in light of 
equations \eqref{eqn:w-ell-real-split} and
\eqref{eqn:fw-real-inert}.
\end{proof}

\begin{example}
Let $\chi_8$ and $\chi_7$ denote the quadratic characters of conductors $8$ and $7$, respectively, and
define $\chi := \chi_8 \chi_7$. Then $S_1(56,\chi)$ is one-dimensional and spanned by the form
\[ g = q - q^2 + q^4 - q^7 - q^8 - q^9 + q^{14} + q^{16} + q^{18} + 2q^{23} - \cdots .\]
We take $p  = 23$, an irregular prime for $g$, and compute a basis for the generalised eigenspace modulo
$(p^{15},q^{3000})$. The two dimensional space complementary to the classical space
has a natural basis
\[ \tilde{g}_1^\flat = q^3 + \cdots + 0 \cdot q^{p} + \cdots \quad \mbox{and} \quad
\tilde{g}_2^\flat = q^2 + 0 \cdot q^3 + \cdots +  \cdots + 0 \cdot q^{p} + \cdots.\]
Take
\[ g_1^\flat := \frac{1}{2} \cdot \log_p(u_2) \cdot \tilde{g}_1^\flat \quad \mbox{and} \quad 
g_2^\flat := \log_p(u_3) \cdot \tilde{g}_2^\flat. \]
Here $u_\ell$, $\ell = 2$ and $3$, denotes a fundamental $\ell$-unit of norm $1$ in
$\Q(\sqrt{-7})$ and $\Q(\sqrt{-56})$, respectively.
One finds that the coefficients at primes $\ell$ which are split in $\Q(\sqrt{8})$ of both forms 
$g_1^\flat$ and $g_2^\flat$ are zero. 
At inert
primes the coefficients of $g_1^\flat$ are the logarithms of fundamental $\ell$-units of
norm $1$ in $\Q(\sqrt{-7})$, and those of  $g_2^\flat$ are the logarithms of
fundamental $\ell$-units of norm $1$ in $\Q(\sqrt{-56})$ (such logarithms are
well-defined up to sign; interestingly, the forms $g_j^\flat$ 
single out a consistent choice of signs). 
\end{example}

 \end{document}